\documentclass{mfat}
\pagespan{356}{380}

 \newtheorem{thm}{Theorem}[section]
 \newtheorem{cor}[thm]{Corollary}
 \newtheorem{lem}[thm]{Lemma}
 \newtheorem{prop}[thm]{Proposition}
 \theoremstyle{definition}
 \newtheorem{defn}[thm]{Definition}
 
 \theoremstyle{remark}
 \newtheorem{rem}[thm]{Remark}
 
 \numberwithin{equation}{section}

\begin{document}

\title[One generalization of the classical moment problem]
      {One generalization \\ of the classical moment problem}

\author{Volodymyr Tesko}
\address{Institute of Mathematics, National Academy of Sciences of Ukraine, 3 Tereshchenkivs'ka,
Kyiv, 01601, Ukraine}
\email{tesko@imath.kiev.ua}

\thanks{The work is partially supported by the DAAD, Project A/11/05321}

\subjclass[2000]{Primary 44A60, 47A57}
\date{11/07/2011}
\keywords{Convolution, positive functional, moment problem,
projection spectral theorem, Sheffer polynomials}

\begin{abstract}
  Let $\ast_P$ be a product on $l_{\rm{fin}}$ (a space of all finite sequences)
  associated with a fixed family $(P_n)_{n=0}^{\infty}$ of real polynomials on
  $\mathbb{R}$. In this article, using methods from the theory of
  generalized eigenvector expansion, we investigate moment-type properties
  of $\ast_P$-positive functionals on $l_{\rm{fin}}$.

  If $(P_n)_{n=0}^{\infty}$ is a family of the Newton polynomials
  $P_n(x)=\prod_{i=0}^{n-1}(x-i)$ then the corresponding product
  $\star=\ast_P$ is an analog of the so-called Kondratiev--Kuna
  convolution on a ``Fock space''. We get an explicit expression for
  the product $\star$ and establish a connection between
  $\star$-positive functionals on $l_{\rm{fin}}$ and a one-dimensional
  analog of the Bogoliubov generating functionals (the classical
  Bogoliubov functionals are defined correlation functions for
  statistical mechanics systems).
\end{abstract}

\maketitle

\section{Introduction}

It is well known that the classical moment problem can be viewed as a
theory of spectral representations of positive functionals on some
classical commutative algebra with involution. Namely, let
$l_{\rm{fin}}$ be a space of all finite sequences
$f=(f_0,\ldots,f_n,0,0,\ldots)$ of complex numbers $f_n$ and $\ast$
denote the Cauchy product on $l_{\rm fin}$, i.e.,
    \begin{equation}\label{eq.convolution.classic.intr}
          (f\ast g)_n:=\sum_{i+j=n}f_{i}g_{j}=\sum_{k=0}^{n}f_{k}g_{n-k}
    \end{equation}
for all $f=(f_n)_{n=0}^{\infty},g=(g_n)_{n=0}^{\infty}\in l_{\rm
fin}$. The space $l_{\rm{fin}}$ endowed with the product $\ast$ is a
commutative algebra with the involution
$f=(f_n)_{n=0}^{\infty}\mapsto \bar{f}:=(\bar{f}_n)_{n=0}^{\infty}$.

The classical moment problem is formulated as follows: {\it for a
given sequence $(\tau_n)_{n=0}^{\infty}$ of real numbers $\tau_n$
when does there exist a non-negative finite Borel measure $\mu$ on
$\mathbb{R}$ such that}
\begin{equation}\label{eq.class.moment.problem}
   \tau_n=\int_{\mathbb{R}}x^n\,d\mu(x),\quad n\in\mathbb{N}_0:=\{0,1,\ldots\}\,?
\end{equation}
The answer is the following: {\it integral representation
(\ref{eq.class.moment.problem}) holds if and only if
$\tau=(\tau_n)_{n=0}^{\infty}$ is a $\ast$-positive functional (more
exactly, non-negative) on $l_{\rm{fin}}$, i.e.,}
$$
      \tau(f\ast\bar{f})=\sum_{j,k=0}^{\infty}\tau_{j+k}f_j\bar{f}_k\geq
      0,\quad f=(f_n)_{n=0}^{\infty}\in l_{\rm{fin}}.
$$

In 
this article an essential role will be played by Yu.~M.~Berezansky's
method \cite{B65} of obtaining representation
(\ref{eq.class.moment.problem}),
which goes back to the works of
M.~G.~Krein \cite{Krein46, Krein48}.
This method is based on the theory of generalized eigenfunction
expansion for selfadjoint operators and,
in its modern version \cite{B02a}, can be formulated as follows. Let
($l_{\rm fin}, \ast$) be an algebra as above  and
$\tau=(\tau_n)_{n=0}^{\infty}$ be a given $\ast$-positive
functional. This pair can be associated, in a usual way, with a
Hilbert space
$H_{\tau}$ generated by
a quasiscalar product
\begin{equation*}
         (f,g)_{H_{\tau}}:=\tau(f\ast\bar{g}),\quad
         f,g\in l_{\rm{fin}}.
\end{equation*}
In this space $H_{\tau}$ the translation operator
\begin{equation}\label{eq.0-0002}
   Jf:=\delta_1\ast f=(0,f_0,f_1,\ldots),\quad f=(f_n)_{n=0}^{\infty}\in l_{\rm
   fin}
\end{equation}
(here $\delta_1:=(0,1,0,0,\ldots)\in l_{\rm{fin}}$) is Hermitian
with equal defect numbers. Therefore it follows from the theory of
generalized eigenfunction expansions that there exists a
non-negative finite Borel measure $\mu$ on $\mathbb{R}$ (spectral
measure) such that  $(x^n)_{n=0}^{\infty}$ is a generalized
eigenvector of the operator $J$ (more exactly, to its selfadjoint
extension) with an eigenvalue $x\in\mathbb{R}$ and the mapping
(Fourier transform)
\begin{equation*}
   H_{\tau}\supset l_{\rm fin}\ni f\mapsto(If)(x):=
   \sum_{n=0}^{\infty}f_n x^n\in L^2(\mathbb{R},\mu)
\end{equation*}
is well-defined and isometric, i.e.,
$$
    (f,g)_{H_{\tau}}=\int_{\mathbb{R}}(If)(x)\overline{(Ig)(x)}d\mu(x),
    \quad f,g\in l_{\rm fin}.
$$
As a consequence of this Parseval equality, we immediately get
representation (\ref{eq.class.moment.problem}):
$$
   \tau_n=(\delta_n, \delta_0)_{H_{\tau}}=\int_{\mathbb{R}}x^n
   \,d\mu(x),
$$
where $\delta_n=(\delta_{nj})_{j=0}^{\infty}$ denotes a
$\delta$-sequence.

Among the advantages of Yu.~M.~Berezansky's method is that this
method admits broad generalizations which give a possibility to
investigate the following moment problems:
strong Hamburger, trigonometric, complex, matrix and different
many-dimensional analogs of them, including infinite-dimensional
cases (in many-dimensional situation it is necessary to investigate
the commuting families of Jacobi type operators), see \cite{B65,
BK88, BKKL99, B02a, B03, BD05, BD06, BM07, BD10} for more detailed
presentation.


By analogy with the above described way of obtaining representation
(\ref{eq.class.moment.problem}), 
we can get moment-type representations in the case when a family
$(x^n)_{n=0}^{\infty}$ of the monomials is replaced by a family
$(P_n)_{n=0}^{\infty}$ of polynomials $P_n:\mathbb{R}\to\mathbb{R}$
(each $P_n$ has a degree $n$). In this situation, instead of the
Cauchy product $\ast$ (\ref{eq.convolution.classic.intr}), it is
necessary to use the product
\begin{equation*}
    f\ast_P g:=I_P^{-1}(I_Pf\cdot I_Pg),\quad (I_Pf)(x):=\sum_{n=0}^{\infty}f_n
    P_n(x), \quad f,g\in l_{\rm fin},
\end{equation*}
generated by the polynomials $P_n(x)$  and, instead of
(\ref{eq.0-0002}), the operator
\begin{equation*}
   J_Pf:=\delta_1\ast_P f,\quad f=(f_n)_{n=0}^{\infty}\in l_{\rm
   fin}
\end{equation*}
(clearly, if $P_n(x)=x^n$ then $\ast=\ast_P$ and $J_P=J$). Let
$H_{\tau}=H_{\tau,P}$ denotes a Hilbert space associated with the
quasiscalar product $(f,g)_{H_{\tau}}:=\tau(f\ast_P\bar{g})$. It can
be shown that for a given $\ast_P$-positive functional
$\tau=(\tau_n)_{n=0}^{\infty}$ on $l_{\rm{fin}}$ there exists a
non-negative finite Borel measure $\mu$ on $\mathbb{R}$ such that
$(P_n)_{n=0}^{\infty}$ is a generalized eigenvector of the operator
$J_P$ ($J_P$ acts in $H_{\tau}$) with an eigenvalue $x\in\mathbb{R}$
and the mapping
\begin{equation*}
   H_{\tau}\supset l_{\rm fin}\ni f\mapsto(I_Pf)(x):=
   \sum_{n=0}^{\infty}f_n P_n(x)\in L^2(\mathbb{R},\mu)
\end{equation*}
is a Fourier transform. The corresponding Parseval equality gives
the moment-type representation
$$
   \tau_n=\int_{\mathbb{R}}P_n(x)\,d\mu(x),\quad n\in \mathbb{N}_0.
$$

The described way of proving the latter representation is given
below in Section~\ref{s.moment-problem}. Let us mention that the
idea of using the theory of generalized eigenvector expansion in a
similar context is not new, see \cite{B02a} for details.

If $(P_n)_{n=0}^{\infty}$ is a family of the so-called Newton
polynomials $P_n(x)=(x)_n:=\prod_{i=0}^{n-1}(x-i)$ then the
corresponding product $\star:=\ast_P$ on $l_{\rm{fin}}$ is an analog
of the so-called Kondratiev--Kuna convolution on a ``Fock space''.
Formula (\ref{eq.convolution.k-k}) in Section~\ref{s.convolution}
gives an explicit expression for the product $\star$. We refer to
\cite{KK02} for the definition and properties of the
Kondratiev--Kuna convolution on a ``Fock space'', see also
Subsection~\ref{s-s.correlations-functions}.

In this article we also study the following problem: {\it for a
given sequence $(\tau_n)_{n=0}^{\infty}$ of real numbers $\tau_n$
when does there exist a non-negative finite Borel measure $\mu$ on
$\mathbb{R}$ such that a Laplace transform
$l_{\mu}(\lambda):=\int_{\mathbb{R}}e^{x\lambda}\,d\mu(x)$ is
analytic in a neighborhood of zero in $\mathbb{C}$  and
$\tau_n=\int_{\mathbb{R}}P_n(x)\,d\mu(x)$ for all $n\in
\mathbb{N}_0$}?

We give an answer on this problem in
Section~\ref{s.analitik-moment-problem} for the case of the
so-called Sheffer polynomials (i.e., polynomials with generating
function of exponential type). The monomials and Newton polynomials
are examples of the Sheffer polynomials. In the case of the
monomials this problem is closely related to the problem of integral
representation of exponentially convex functions (see
Subsection~\ref{s.moment-problem-exp-convex}), in the Newton
polynomials' case this problem is connected with a one-dimensional
analog of the Bogoliubov generating functionals, i.e., with
functions $B:\mathcal{U}\to\mathbb{C}$ ($\mathcal{U}$ is a
neighborhood of $0\in\mathbb{C}$) of such type
\begin{equation*}
    B(\lambda)=\int_{\mathbb{R}}(1+\lambda)^xd\mu(x)=\int_{\mathbb{R}}e^{x\log(1+\lambda)}d\mu(x),
    \quad \lambda\in \mathcal{U},
\end{equation*}
where $\mu$ is a certain non-negative finite Borel measure on
$\mathbb{R}$ (see Subsection~\ref{s.newton-bogolubov}). Note that
$e^{x\log(1+\lambda)}$ is a generating function for the Newton
polynomials $(x)_n$,
\begin{equation*}
     e^{x\log(1+\lambda)}=\sum_{n=0}^{\infty}\frac{\lambda^n}{n!}(x)_n,\quad
     |\lambda|<1.
\end{equation*}
We stress that the classical Bogoliubov functionals were introduced
by N.~N.~Bogoliubov in \cite{Bogoliubov} to define correlation
functions for statistical mechanics systems.

The last part of this article (Section~\ref{s.infinite-dim-case}) is
related, on the one hand, to the infinite-dimensional generalization
of the classical moment problem (\ref{eq.class.moment.problem}) and,
on the other hand, to some tasks of statistical physics. The main
purpose of this section is to explain the motivation of this work
and present a few known examples of the results and some open
problems in the case of functions of infinite many variables.

\section{Preliminaries}

\subsection{Projection spectral theorem}

In this subsection we recall some results concerning the projection
spectral theorem and the quasianalytic criterion of selfadjointness
of operators (for a detailed explanation see e.g. books \cite{B65,
BK88, BUSH}).

Let $\mathcal{H}$ be a complex separable Hilbert space and
\begin{equation}\label{eq.chain}
    \mathcal{H}_-\supset \mathcal{H} \supset \mathcal{H}_+ \supset
    \mathcal{D}
\end{equation}
be a fixed rigging of $\mathcal{H}$. We suppose that $\mathcal{H}_+$
is a Hilbert space which is topologically (i.e., densely and
continuously) and quasinuclearly (i.e., the inclusion operator is of
Hilbert--Schmidt type) embedded into $\mathcal{H}$, $\mathcal{H}_-$
is the dual of $\mathcal{H}_+$ with respect to the zero space
$\mathcal{H}$ (with the pairing $\langle\cdot\,
,\cdot\rangle_{\mathcal{H}}$), and $\mathcal{D}$ is a linear,
separable, topological space that is topologically embedded into
$\mathcal{H}_+$.

The following projection spectral theorem holds (see \cite{B65},
Ch.~5; \cite{BK88}, Ch.~3; \cite{BUSH}, Ch.~15).

\begin{thm}\label{thm.spectral}
    Let $A$ be a self-adjoint operator defined on ${\rm Dom}(A)$ in $\mathcal{H}$. Assume that
    \begin{itemize}
      \item $A$ is standardly connected with chain
      (\ref{eq.chain}), i.e., $\mathcal{D}\subset{\rm Dom}(A)$ and
      the restriction $A\upharpoonright\mathcal{D}$ of the operator $A$ on $\mathcal{D}$ acts from $\mathcal{D}$
      into $\mathcal{H}_+$ continuously.
      \item $A$ has a strong cyclic vector
      $\Omega$, that is there exists a vector $\Omega\in\mathcal{H}$
      such that $\Omega\in{\rm Dom}(A^n)$ for all $n\in\mathbb{N}$
      and a set $\{A^n\Omega\,|\,n\in\mathbb{N}_0\}$ is total in
      $\mathcal{H}_+$ (i.e., a set ${\rm span}\{A^n\Omega\,|\,n\in\mathbb{N}_0\}$ is dense in
      $\mathcal{H}_+$).
    \end{itemize}
    Then there exists a non-negative finite Borel measure $\mu$ on
    $\mathbb{R}$ (spectral measure, defined on the Borel $\sigma$-algebra $\mathcal{B}(\mathbb{R})$)
    such that
    \begin{itemize}
      \item For $\mu$-almost every $x\in\mathbb{R}$ there is a
      unique vector $\xi(x)\in\mathcal{H}_-$ (the so-called generalized eigenvector of  $A$ with an eigenvalue $x$) such that
      $$
         \langle\xi(x),Af\rangle_{\mathcal{H}}=x\langle\xi(x),f\rangle_{\mathcal{H}},
         \quad f\in\mathcal{D}.
      $$
      \item The mapping
      \begin{equation}\label{eq.self.unitary}
         \mathcal{H}\supset\mathcal{D}\ni f\mapsto(I_Af)(\cdot):=\langle
         f,\xi(\cdot)\rangle_{\mathcal{H}}\in L^2(\mathbb{R},\mu)
      \end{equation}
      is well-defined and isometric, i.e.,
      \begin{equation}\label{eq.pars.equality}
         (f,g)_{\mathcal{H}}=\int_{\mathbb{R}}(I_Af)(x)\overline{(I_Ag)(x)}\,d\mu(x),
         \quad f,g\in\mathcal{H}.
      \end{equation}
      Extending the mapping $I_A$ by continuity to the
      whole space $\mathcal{H}$ we obtain an isometric operator
      $I_A:\mathcal{H}\to L^2(\mathbb{R},\mu)$.
    \end{itemize}
\end{thm}

\begin{rem}\label{r.Fouirier}
     Let an operator $A$ satisfies all assumptions of Theorem~\ref{thm.spectral}
     and, moreover, the closure of $A\upharpoonright\mathcal{D}$ in $H$ coincides with
     $A$. Then by a well-known fact (see, e.g., \cite{BUSH}, Ch. 15, \S~3) the extension (by continuity) of mapping
     (\ref{eq.self.unitary}) is a unitary operator $I_A:\mathcal{H}\to L^2(\mathbb{R},\mu)$ acting from the
     whole space $\mathcal{H}$ onto the whole space $L^2(\mathbb{R},\mu)$. The
      image of $A$ under $I_A$ is the operator of multiplication by
      $x$ in $L^2(\mathbb{R},\mu)$.
\end{rem}

Let us recall the quasianalytic criterion of self-adjointness. For a
Hermitian operator $A$ defined on ${\rm Dom}(A)$ in $\mathcal{H}$, a
vector $f\in\bigcap_{n=1}^{\infty}{\rm Dom}(A^n)$ is called
quasianalytic if
\begin{equation}\label{eq.quasianalytic}
   \sum_{n=1}^{\infty}\frac{1}{\sqrt[n]{\|A^nf\|_{\mathcal{H}}}}=\infty.
\end{equation}

\begin{thm}\label{thm.quasianalytic}
   A Hermitian operator $A$ in $\mathcal{H}$ is essentially
   self-adjoint if and only if the space $\mathcal{H}$ contains a total set of
   quasianalytic vectors.
\end{thm}

Versions of this theorem are published in \cite{Nussbaum1,
Nussbaum2}, see also \cite{B65}, Ch.~8, \S~5. For the given form of
it, see \cite{BUSH}, Ch.~13, \S~9.


\subsection{Spaces and riggings}

Denote by $\mathbb{C}^{\infty}$ a linear space of all sequences
$f=(f_n)_{n=0}^{\infty}$ of complex numbers $f_n\in\mathbb{C}$, and
by $l_{\rm{fin}}$ its linear subspace consisting of finite sequences
$f=(f_0,\ldots,f_n,0,0,\ldots)$. Henceforth, we will denote by
$\delta_n$ the $\delta$-sequence,
\begin{equation}\label{eq.delta_n}
\delta_n=(\delta_{nj})_{j=0}^{\infty}=(\underbrace{0,\ldots,0}_{n\,\text{times}},1,0,0,\ldots).
\end{equation}
Then each vector $f$ from $l_{\rm{fin}}$ can be interpreted as a
finite sum $\sum_{n=0}^{\infty}f_n\delta_n$.


For a fixed weight $p=(p_n)_{n=0}^{\infty},\;p_n>0,$ we denote by
$$
   l^2(p):=\Big\{f=(f_n)_{n=0}^{\infty}\in \mathbb{C}^{\infty}\,\Big|\,
   \|f\|_{l^2(p)}^{2}:=\sum_{n=0}^{\infty}|f_n|^2p_n<\infty\Big\}
$$
the $l^2$-type space with a corresponding scalar product
$(\cdot\,,\cdot)_{l^2(p)}$. In the case of the weight
$1=(1,1,\ldots)$ we will use a standard notation $l^2:=l^2(1)$.

Let $p=(p_n)_{n=0}^{\infty},\;p_n\geq 1$. Then the space $l^2(p)$ is
densely and continuously embedded into the space $l^2$ and therefore
one can construct the chain (the rigging of $l^2$)
\begin{equation}\label{eq.riggin.moment}
      l^2(p^{-1}) \supset  l^2   \supset l^2(p) \supset l_{\rm{fin}},
\end{equation}
where $p^{-1}:=(p^{-1}_n)_{n=0}^{\infty}$ and
$l^2(p^{-1})=(l^2(p))'$ is the dual space of $l^2(p)$ with respect
to the zero space $l^2$. Denote by $\langle\cdot\, ,
\cdot\rangle_{l^2}$ the dual pairing between elements of
$l^2(p^{-1})$ and $l^2(p)$ inducted by the scalar product in $l^2$,
i.e., 
$$
   \langle \xi,g\rangle_{l^2}:=\sum_{n=0}^{\infty}\xi_n\bar{g}_n,
   \quad \xi\in l^2(p^{-1}),\quad g\in l^2(p).
$$

Together with (\ref{eq.riggin.moment}), we consider a rigging of
$l^2$ connected with a special weight $p$. Namely, for each
$q\in{\mathbb N}$, we set
$$
      \gamma(q)=((n!)^22^{qn})_{n=0}^{\infty}
$$
and introduce the so-called  Kondratiev-type $l^2$-spaces
$$
   l^2(\gamma(q))\quad\text{and}\quad l^2_+:=\mathop{\rm pr\,lim}_{q\in{\mathbb
   N}}l^2(\gamma(q)).
$$
Then the dual spaces of $l^2(\gamma(q))$ and $l^2_+$ with respect to
the zero space $l^2$ are
$$
   l^2(\gamma^{-1}(q))\quad\text{and}\quad
   l^2_-:=(l^2_+)'=\mathop{\rm ind\,lim}_{q\in{\mathbb N}}l^2(\gamma^{-1}(q))
$$
respectively (here $\gamma_n^{-1}(q)=(n!)^{-2}2^{-qn}$). Thus, we
get a rigging
\begin{equation*}
   \mathbb{C}^{\infty}=l_{\rm{fin}}' \supset l^2_- \supset l^2(\gamma^{-1}(q)) \supset  l^2   \supset l^2(\gamma(q)) \supset
   l^2_+ \supset l_{\rm{fin}}.
\end{equation*}
Here we identify, in the usual way, the space $\mathbb{C}^{\infty}$
with the space $l_{\rm{fin}}'$ of all linear functionals on
$l_{\rm{fin}}$. In the sequel we won't distinguish
$\mathbb{C}^{\infty}$ and $l_{\rm{fin}}'$.

Now we recall one important property of the space $l^2_-$. Denote by
${\rm Hol}_0(\mathbb{C})$ a set of all (germs of) functions
$\phi:\mathbb{C}\to\mathbb{C}$ that are holomorphic at $0\in
\mathbb{C}$ and, for each $\xi=(\xi_n)_{n=0}^{\infty}\in l_{-}^2$,
define the so-called {\it $S$-transform} by the formula
\begin{equation*}
(S\xi)(\lambda):=\sum_{n=0}^{\infty}\frac{\lambda^n}{n!}\xi_n, \quad
\lambda\in\mathcal{U},
\end{equation*}
where $\mathcal{U}$ is a (depending on $\xi$) neighborhood of
$0\in\mathbb{C}$.

The following result shows that each vector $\xi$ from $l_{-}^2$ is
uniquely determined by its $S$-transform (see \cite{KLS96} for the
infinite dimensional analogue of this fact).

\begin{thm}\label{t.s-transform}
  The $S$-transform
  \begin{equation*}\label{eq.s-transform}
     S:l_{-}^2\to {\rm Hol}_0(\mathbb{C}),\quad \xi=(\xi_n)_{n=0}^{\infty}
     \mapsto (S\xi)(\lambda):=\sum_{n=0}^{\infty}\frac{\lambda^n}{n!}\xi_n,
  \end{equation*}
  is a one-to-one map between  $l_{-}^2$ and ${\rm Hol}_0(\mathbb{C})$.
\end{thm}

\begin{proof}
  Let $\xi=(\xi_n)_{n=0}^{\infty}\in l_{-}^2$, i.e., there exists
  $q\in\mathbb{N}$ such that
  $$
     \xi\in l^2(\gamma^{-1}(q))\quad
     \text{or, equivalently,}\quad \sum_{n=0}^{\infty}|\xi_n|^22^{-qn}(n!)^{-2}<\infty.
  $$
  Using the Cauchy–-Bunyakovsky–-Schwarz inequality, for $|\lambda|<2^{-\frac{q}{2}}$, we get
  \begin{align*}
     |(S\xi)(\lambda)|
     &
     =\Big|\sum_{n=0}^{\infty}\frac{\lambda^n}{n!}\xi_n\Big|
     \leq \sum_{n=0}^{\infty}\frac{|\lambda|^n}{n!}|\xi_n|\\
     &
     \leq \Big(\sum_{n=0}^{\infty}|\lambda|^{2n}2^{qn}\Big)^{\frac{1}{2}}
     \Big(\sum_{n=0}^{\infty}|\xi_n|^{2}2^{-qn}(n!)^{-2}\Big)^{\frac{1}{2}}<\infty.
  \end{align*}
  Thus, $S\xi\in {\rm Hol}_0(\mathbb{C})$.

  For the converse, suppose that $\phi\in {\rm Hol}_0(\mathbb{C})$, that is there exists
  $r>0$ such that the function $\phi$ admits the representation
  $$
     \phi(\lambda)=\sum_{n=0}^{\infty}\frac{\lambda^n}{n!}\xi_n<\infty,\quad
     |\lambda|<r,
  $$
  with
  $$
     \xi_n=\frac{d^n\phi}{d\lambda^n}(\lambda)\Big|_{\lambda=0}
     =\frac{n!}{2\pi i}\oint_{|\zeta|=r_0}\frac{\phi(\zeta)}{\zeta^{n+1}}d\zeta,
   \quad 0<r_0<r.
  $$
  As a consequence of the latter integral representation, for some
  $C>0$, we get
  $$
   |\xi_n|\leq n!C^{n+1},\quad n\in\mathbb{N}_0.
  $$
  Choosing
  $q\in\mathbb{N}$ in such way that $C^22^{-q}<1$, we obtain
  $$
      \sum_{n=0}^{\infty}|\xi_n|^22^{-qn}(n!)^{-2}
      \leq C^2\sum_{n=0}^{\infty}\frac{C^{2n}}{2^{qn}}<\infty.
  $$
  So, $\xi:=(\xi_n)_{n=0}^{\infty}\in l_{-}^2$ and $S\xi=\phi$.

  The fact that
  ${\rm Ker}(S):=\{\xi\in l_{-}^2\,|\,S\xi=0\}=\{0\}$ is obvious.
\end{proof}

\begin{cor}\label{cor.hol}
 A sequences $\xi=(\xi_n)_{n=0}^{\infty}\in \mathbb{C}^{\infty}$
 belongs to the space $l_{-}^2$ if and only if there exists a constant $C>0$ such that
 $$
    |\xi_n|\leq n!C^{n+1},\quad n\in\mathbb{N}_0.
 $$
\end{cor}



\section{Convolutions on the space of finite sequences}\label{s.convolution}

\subsection{Definition and properties of convolutions}

Let $(P_n)_{n=0}^{\infty}$ be a fixed family of real-valued 
polynomials $P_n:\mathbb{R}\to\mathbb{R}$ such that 
each $P_n$ has a degree $n$. Thus, $(P_n)_{n=0}^{\infty}$ is a
linear basis in the space $\mathcal{P}:=\mathbb{C}[x]$ of all
complex-valued polynomials $F:\mathbb{R}\to\mathbb{C}$.

Define a convolution (product) $\ast_P$ on the space $l_{\rm{fin}}$
by setting
\begin{equation}\label{eq.convolution.g}
\begin{gathered}
    f\ast_P g:=I_P^{-1}(I_Pf\cdot I_Pg),\quad f,g\in
l_{\rm{fin}},
\end{gathered}
\end{equation}
where
\begin{equation}\label{eq.fourier.polynom}
      I_P:l_{\rm{fin}}\to \mathbb{C}[x],\quad
    f=(f_n)_{n=0}^{\infty}\mapsto
    (I_Pf)(x):=\sum_{n=0}^{\infty}f_n P_n(x),
\end{equation}
is a natural bijection between $l_{\rm{fin}}$ and $\mathbb{C}[x]$.
The space $l_{\rm{fin}}$ with product $\ast_P$ becomes a commutative
algebra $\mathcal{A}$ with the unity $\delta_0=\{1,0,0,\ldots\}$ and
the involution
\begin{equation}\label{eq.involution}
l_{\rm{fin}}\ni f=(f_n)_{n=0}^{\infty}\mapsto
\bar{f}:=(\bar{f}_n)_{n=0}^{\infty}\in l_{\rm{fin}},
\end{equation}
where $\bar{f}_n$ denotes the complex conjugation. Clearly, choosing
different bases in the space $\mathbb{C}[x]$ we obtain different
products in the space $l_{\rm{fin}}$.

In the space $\mathbb{C}[x]$ we introduce a scalar product by
setting
$$
   (F,G)_{\mathcal{P}}:=(I_P^{-1}F,I_P^{-1}G)_{l^2}=\sum_{n=0}^{\infty}f_n\bar{g}_n,
$$
$$
   F(\cdot)=\sum_{n=0}^{\infty}f_nP_n(\cdot),
   \quad G(\cdot)=\sum_{n=0}^{\infty}g_nP_n(\cdot)\in \mathbb{C}[x].
$$
The sequence $(P_n)_{n=0}^{\infty}$ makes an orthonormal basis in
$\mathbb{C}[x]$ and therefore each polynomial $F\in\mathbb{C}[x]$
admits the representation
\begin{equation}\label{eq.scpar}
   F(x)=\sum_{n=0}^{\infty}(F,P_n)_{\mathcal{P}}P_n(x),\quad
   x\in\mathbb{R}
\end{equation}
(note that $(F,P_n)_{\mathcal{P}}=0$ for $n$ greater than the degree
of $F$).

The following result holds, see also \cite{B02a}.
\begin{lem}
    For all $f,g\in l_{\rm{fin}}$ and $n\in\mathbb{N}_0:=\{0,1,\ldots\}$ we have
    \begin{equation}\label{eq.multiplication}
       (f\ast_Pg)_n=\sum_{j,k=0}^{\infty}f_jg_k(P_jP_k,P_n)_{\mathcal{P}}.
    \end{equation}
\end{lem}

\begin{proof}
   Using formulas (\ref{eq.convolution.g}), (\ref{eq.fourier.polynom}) and (\ref{eq.scpar}),
   for all $f,g\in l_{\rm{fin}}$ and $x\in\mathbb{R}$, we get
   \begin{align*}
         (I_P(f\ast_P g))(x)
        &
         =(I_Pf)(x)\cdot(I_Pg)(x) =\sum_{n=0}^{\infty}(I_Pf\cdot
         I_Pg,P_n)_{\mathcal{P}}P_n(x)\\
        &
         =\sum_{n=0}^{\infty}\Big(\sum_{j,k=0}^{\infty}f_jg_k(P_jP_k,P_n)_{\mathcal{P}}\Big)P_n(x)
         =\sum_{n=0}^{\infty}(f\ast_P g)_nP_n(x).
   \end{align*}
   So, formula (\ref{eq.multiplication}) takes place.
\end{proof}



\subsection{Examples}

In some special cases we can count a more explicit expression of
pro\-duct~(\ref{eq.multiplication}). Let us consider two examples.

1) Let $P_n(x)=x^n$ be a monomial. Then $\ast_P=\ast$ is an ordinary
convolution {\it{\rm(}Cauchy product{\rm)} $\ast$} of two sequences
$f=(f_n)_{n=0}^{\infty},g=(g_n)_{n=0}^{\infty}\in l_{\rm fin}$
    \begin{equation}\label{eq.convolution.classic}
          (f\ast_P g)_n=(f\ast g)_n
          =\sum_{i+j=n}f_{i}g_{j}=\sum_{k=0}^{\infty}f_{k}g_{n-k}.
    \end{equation}

    This fact is a direct consequence of (\ref{eq.multiplication}) and
    the following formula:
    $$
       (P_jP_k,P_n)_{\mathcal{P}}=(x^jx^k,x^n)_{\mathcal{P}}=
       (x^{j+k},x^n)_{\mathcal{P}}=\delta_{j+k,n}.
    $$

2) Let $P_n(x)=(x)_n$ (here $(x)_n$ denotes the Pochhammer symbol)
be the so-called {\it Newton {\rm(}or binomial{\rm)} polynomial}. By
definition
\begin{equation*}
        P_n(x)=(x)_n:=
        \begin{cases}
             1, &  \hbox{if}\quad n=0, \\
             x(x-1)\cdots(x-n+1), & \hbox{if}\quad n\in\mathbb{N}.
        \end{cases}
\end{equation*}
In terms of Gamma function, we have
$$
   (x)_n=\frac{\Gamma(x+1)}{\Gamma(x-n+1)},\quad n\in\mathbb{N}.
$$

Note that $(x)_n,\,n\in\mathbb{N}_0$, is an example of {\it Sheffer
polynomials}, see Section~\ref{s.sheffer.polynom} for details. The
corresponding generating function of $(x)_n$ has the form
\begin{equation}\label{eq.newton-def-onedim}
    P(x,\lambda):=(1+\lambda)^x=e^{x\log(1+\lambda)}=\sum_{n=0}^{\infty}\frac{\lambda^n}{n!}(x)_n,\quad
     |\lambda|<1.
\end{equation}

In the case $P_n(x)=(x)_n$ we will denote convolution
(\ref{eq.convolution.g}) by $\star:=\ast_P$.

\begin{thm}
    For all $f=(f_n)_{n=0}^{\infty},g=(g_n)_{n=0}^{\infty}\in l_{\rm fin}$ and $n\in\mathbb{N}_0$ we have
    \begin{equation}\label{eq.convolution.k-k}
    \begin{split}
          (f\ast_P g)_n=(f\star g)_n
          =\sum_{i+j+k=n}\frac{(i+j)!(i+k)!}{i!k!j!}
          f_{i+j}g_{j+k}.
    \end{split}
    \end{equation}
\end{thm}

\begin{proof}
    Due to (\ref{eq.multiplication}) it suffices to show that
    \begin{equation}\label{eq.mult.pr}
        ((x)_j(x)_k,(x)_n)_{\mathcal{P}}=
        \begin{cases}
             \displaystyle\frac{j!k!}{(n-j)!(n-k)!(j+k-n)!}, &  j,k\in\{0,\ldots,n\}, j+k\geq n, \\
             0, & \hbox{otherwise.}
        \end{cases}
    \end{equation}

    Using (\ref{eq.newton-def-onedim}) and the multinomial formula
    $$
       (a_1+a_2+\cdots +a_m)^n=\sum_{k_1+k_2+\cdots +k_m=n}\frac{n!}{k_1!k_2!\cdots
       k_m!}\prod_{1\leq i\leq m}a_{i}^{k_i},
    $$
    for all $|\lambda|, |\mu|<\varepsilon$ ($\varepsilon>0$ is small enough), we get
    \begin{equation}\label{eq.pr.1}
    \begin{split}
       e^{x\log(1+\lambda)}e^{x\log(1+\mu)}
      &=
       e^{x\log(1+\lambda+\mu+\lambda\mu)}=\sum_{n=0}^{\infty}\frac{(\lambda+\mu+\lambda\mu)^n}{n!}(x)_n\\
      &=
       \sum_{n=0}^{\infty}\sum_{i+j+k=n}\frac{\lambda^{i+j}\mu^{k+j}}{i!j!k!}(x)_n.
    \end{split}
    \end{equation}

    On the other hand, taking into account the formula
    $$
        (x)_j(x)_k=\sum_{n=0}^{\infty}((x)_j(x)_k,(x)_n)_{\mathcal{P}}(x)_n,\quad
        x\in\mathbb{R},
    $$
    for all $|\lambda|, |\mu|<\varepsilon$, we obtain
    \begin{equation}\label{eq.pr.2}
    \begin{split}
        e^{x\log(1+\lambda)}e^{x\log(1+\mu)}
       & =\sum_{j,k=0}^{\infty}\frac{\lambda^j}{j!}\frac{\mu^k}{k!}(x)_j(x)_k\\
       &
        =\sum_{n=0}^{\infty}\sum_{j,k=0}^{\infty}
        \frac{\lambda^{j}\mu^{k}}{j!k!}((x)_j(x)_k,(x)_n)_{\mathcal{P}}(x)_n.
    \end{split}
    \end{equation}

   Comparing the coefficients at $(x)_n$ and then at
   $\lambda^{j}\mu^{k}$ in formulas (\ref{eq.pr.1}) and (\ref{eq.pr.2}) we get equality (\ref{eq.mult.pr}).
\end{proof}

\begin{rem}
   It should be noticed that the product $\star$ is a one-dimensional analog of
   the so-called {\it Kondratiev--Kuna convolution} on a ``Fock space'', see e.g.
   \cite{KK02} and Subsection~\ref{s-s.correlations-functions} below for the definition and properties
   of the Kondratiev--Kuna convolution.
\end{rem}


\section{The moment problem}\label{s.moment-problem}

As above let $(P_n)_{n=0}^{\infty}$ be a fixed family of real-valued
polynomials $P_n\in\mathbb{C}[x]$ such that each $P_n$ has a degree
$n$ and $\mathcal{A}=l_{\rm fin}$ be a commutative algebra with the
product $\ast_P$ (\ref{eq.convolution.g}).

\begin{defn}
A functional $\tau=(\tau_n)_{n=0}^{\infty}\in \mathbb{C}^{\infty}$
is said to be a {\it moment functional (or, a moment sequences) on}
($\mathcal{A}, \ast_P$)  if there exists a non-negative Borel
measure $\mu$ on $\mathbb{R}$ 
such that
\begin{equation}\label{eq.moment-reprez}
   \tau_n=\int_{\mathbb{R}}P_n(x)\,d\mu(x),\quad n\in \mathbb{N}_0.
\end{equation}
\end{defn}

Obviously, if $\tau=(\tau_n)_{n=0}^{\infty}\in \mathbb{C}^{\infty}$
is a moment functional on ($\mathcal{A}, \ast_P$) then actually
$\tau_n\in\mathbb{R}$ for all $n\in\mathbb{N}_0$ and the measure
$\mu$ from representation (\ref{eq.moment-reprez}) is finite. The
{\it moment problem} on $\mathcal{A}$ is to characterize those
linear functionals $\tau\in\mathbb{C}^{\infty}$ which are moment
functionals. A solution of this problem is given in the next theorem. 

\begin{thm}\label{t.moment-problem}
   $\tau=(\tau_n)_{n=0}^{\infty}\in\mathbb{C}^{\infty}$ is a moment functional on $\mathcal{A}=l_{\rm fin}$  if and only
   if $\tau$ is $\ast_P$-positive (more exactly, non-negative) on $\mathcal{A}$, that is
   \begin{equation}\label{eq.positivity-c}
      \tau(f\ast_P\bar{f})=\sum_{n=0}^{\infty}\tau_n(f\ast_P\bar{f})_n
      =\sum_{n=0}^{\infty}\tau_n\Big(\sum_{j,k=0}^{\infty}f_j\bar{f}_k(P_jP_k,P_n)_{\mathcal{P}}\Big)\geq
      0
   \end{equation}
   for all $f=(f_n)_{n=0}^{\infty}\in\mathcal{A}$.
\end{thm}

A method of proving this result is similar to considerations of
\cite{B02a} and is based on the theory of generalized eigenfunction
expansion. In the case of the classical moment problem this method
was first proposed by Yu.~M.~Berezansky in \cite{B65},~Ch.~8.

\begin{proof}
The necessity of condition (\ref{eq.positivity-c}) is trivial.
Indeed,
\begin{align*}
    \tau(f\ast_P\bar{f})
    &=
    \sum_{n=0}^{\infty}\tau_n(f\ast_P\bar{f})_n
    =\sum_{n=0}^{\infty}\tau_n\Big(\sum_{j,k=0}^{\infty}f_j\bar{f}_k(P_jP_k,P_n)_{\mathcal{P}}\Big)\\
    &=
    \int_{\mathbb{R}}\sum_{j,k=0}^{\infty}f_j\bar{f}_k
    \Big(\sum_{n=0}^{\infty}(P_jP_k,P_n)_{\mathcal{P}}P_n(x)\Big)\,d\mu(x)\\
    &=
    \int_{\mathbb{R}}\Big|\sum_{j=0}^{\infty}f_jP_j(x)\Big|^2\,d\mu(x)\geq
    0,\quad f=(f_n)_{n=0}^{\infty}\in\mathcal{A}.
\end{align*}

For the proof of the sufficiency of condition
(\ref{eq.positivity-c}), we will apply Theorem~\ref{thm.spectral} to
a certain self-adjoint operator connected with our moment problem.

Let $\tau\in\mathbb{C}^{\infty}$ be a positive functional on
$\mathcal{A}$, that is (\ref{eq.positivity-c}) holds. Using this
functional and convolution $\ast_P$ we construct in a standard way a
Hilbert space $H_{\tau}$. Namely, we define $H_{\tau}$ as a Hilbert
space associated with the quasiscalar product
\begin{equation}\label{eq.quasiproduct}
         (f,g)_{H_{\tau}}:=\tau(f\ast_P\bar{g}),\quad
         f,g\in\mathcal{A}.
\end{equation}
For the construction of $H_{\tau}$, at first it is necessary to pass
from $\mathcal{A}$ to the factor space $
   \dot{\mathcal{A}}:=\mathcal{A}/\{f\in \mathcal{A}\,|\,(f,f)_{H_{\tau}}=0\}
$ and then to take the completion of $\dot{\mathcal{A}}$. For
simplicity we will suppose that
$\dot{\mathcal{A}}\equiv\mathcal{A}$, i.e., $(f,f)_{H_{\tau}}=0$ if
and only if $f=0$. Note that an investigation of the general case is
possible but technically it is more complicated (for the
corresponding constructions in the case of the classical moment
problem, see \cite{B65}, Ch. 8, \S~1, Subsect. 4 or in \cite{BK88},
Ch. 5, \S~5, Subsect. 1--3).

For the sake of simplicity we will assume that $P_0(x)=1$ and
$P_1(x)=x$. Using (\ref{eq.convolution.g}) and
(\ref{eq.fourier.polynom}) we define an operator
\begin{equation}\label{eq.j-operator}
   J_P:l_{\rm fin}\to l_{\rm fin},
   \quad J_Pf:=I_{P}^{-1}\mathbb{J}I_{P}=\delta_1\ast_Pf,\quad f\in l_{\rm
   fin},
\end{equation}
where $\delta_1=(0,1,0,0,\ldots)$, $I_P$ is defined by formula
(\ref{eq.fourier.polynom}) and $\mathbb{J}$ is the operator of
multiplication by $x$ in the space $\mathbb{C}[x]$, i.e.,
\begin{equation*}
   (\mathbb{J}F)(x):=P_1(x)F(x)=xF(x),\quad F\in \mathbb{C}[x].
\end{equation*}
The operator $J:l_{\rm fin}\to l_{\rm fin}$ is Hermitian in the
Hilbert space $H_{\tau}$,
$$
   (J_Pf,g)_{H_{\tau}}=\tau(\delta_1\ast_Pf\ast_P\bar{g})=
   \tau(f\ast_P\overline{\delta_1\ast_Pg})=(f,J_Pg)_{H_{\tau}},\quad f,g\in l_{\rm
   fin},
$$
and, moreover, it is real with respect to involution
(\ref{eq.involution}), i.e., $\overline{J_Pf}=J_P\bar{f},\;f\in
l_{\rm fin}$. Therefore, by a theorem of von Neumann $J_P$ has
self-adjoint extensions.

Denote by $A$ a certain self-adjoint extension of $J_P$ on
$H_{\tau}$. We will apply Theorem~\ref{thm.spectral} to this
operator. Now the role of chain (\ref{eq.chain}) will play the
rigging
\begin{equation}\label{eq.chain.moment}
    (l^2(p))_{H_{\tau}}'\supset H_{\tau} \supset l^2(p) \supset
    l_{\rm fin},
\end{equation}
where $(l^2(p))_{H_{\tau}}'=\mathcal{H}_-$ is the negative space
with respect to the positive space $l^2(p)$ and the zero space
$H_{\tau}=\mathcal{H}$. The space $l_{\rm fin}=\mathcal{D}$ is
provided with uniformly finite coordinate-wise convergence, i.e.,
the sequence $\{f^{(j)},j\in\mathbb{N}\}\subset l_{\rm{fin}}$
converge to $f\in l_{\rm{fin}}$ if and only if there exists
$N\in\mathbb{N}$ such that $f_n^{(j)}=0$ for all
$n>N,\;j\in\mathbb{N}$ and $f_n^{(j)}\rightarrow f_n$ as $j\to
\infty$ for all $n\in\mathbb{N}_0$.

\begin{lem}
   There exists a weight $p=(p_n)_{n=0}^{\infty},\;p_n\geq 1$, such
   that the embedding $l^2(p)\hookrightarrow H_{\tau}$ is
   well-defined and quasinuclear.
\end{lem}
\begin{proof}
   Let us set
   \begin{equation}\label{eq.k-kernel}
      K_{jk}:=\sum_{n=0}^{\infty}\tau_n(P_jP_k,P_n)_{\mathcal{P}},
      \quad j,k\in\mathbb{N}_0
   \end{equation}
   (note that for any $j,k\in\mathbb{N}_0$ the sum in (\ref{eq.k-kernel}) is
   finite). Due to (\ref{eq.positivity-c}) the matrix $K=(K_{jk})_{j,k=0}^{\infty}$ is
   nonnegative definite, i.e.,
   \begin{equation*}
      \sum_{j,k=0}^{\infty}K_{jk}f_j\bar{f}_k
      =\sum_{j,k=0}^{\infty}\Big(\sum_{n=0}^{\infty}\tau_n(P_jP_k,P_n)_{\mathcal{P}}\Big)f_j\bar{f}_k
      =\tau(f\ast_P\bar{f})\geq  0,\quad f\in l_{\rm fin}.
   \end{equation*}
   Hence,
   \begin{equation}\label{eq.kosh-bun}
      |K_{jk}|^2\leq K_{jj} K_{kk},\quad j,k\in\mathbb{N}_0.
   \end{equation}
   Let $q=(q_n)_{n=0}^{\infty},\;q_n\geq 1$, be such that
   $\sum_{n=0}^{\infty}K_{nn}q_{n}^{-1}<\infty$. Then from
   (\ref{eq.quasiproduct}), (\ref{eq.positivity-c}) and
   (\ref{eq.kosh-bun}) it follows that, for all $f\in l_{\rm fin}$,
   $$
       \|f\|_{H_{\tau}}^{2}=\tau(f\ast_P\bar{f})=\sum_{j,k=0}^{\infty}K_{jk}f_j\bar{f}_k
       \leq
       \Big(\sum_{j=0}^{\infty}\frac{K_{jj}}{q_j}\Big)\|f\|_{l_2(q)}^{2}.
   $$
   Therefore, $l^2(q)\hookrightarrow H_{\tau}$ topologically. But if
   $\sum_{n=0}^{\infty}q_{n}p_{n}^{-1}<\infty$, then $l^2(p)\hookrightarrow
   l^2(q)$ quasinuclearly. The composition of these two embeddings
   gives that $l^2(p)\hookrightarrow H_{\tau}$ is quasinuclear.
\end{proof}

In what follows we fix a weight $p=(p_n)_{n=0}^{\infty},\;p_n\geq
1$, such that the embedding $l^2(p)\hookrightarrow H_{\tau}$ is
quasinuclear. It is clear that the operator $A$ is standardly
connected with chain (\ref{eq.chain.moment}). Let us show that the
vector $\Omega=\delta_0=(1,0,0,\ldots)\in l_{\rm fin}$ is a strong
cyclic vector for $A$.

To this end, it suffices to show that $
   {\rm span}\{A^n\Omega\;|\;n\in\mathbb{N}_0\}=l_{\rm fin}.
$ But this is evidently true, since
$I_P:l_{\rm{fin}}\to\mathbb{C}[x]$ is bijection, ${\rm
span}\{x^n\;|\;n\in\mathbb{N}_0\}=\mathbb{C}[x]$ and by
(\ref{eq.j-operator})
$$
   A^n\Omega=J^n\delta_0=I_{P}^{-1}(x^n),\quad n\in\mathbb{N}_0.
$$

So, the operator $A$ satisfies all assumptions of
Theorem~\ref{thm.spectral}. Let $\mu$ be the corresponding spectral
measure of $A$ and $\xi(x)\in (l^2(p))_{H_{\tau}}'$ be the
generalized eigenvector of $A$ with an eigenvalue $x\in\mathbb{R}$.
According to Theorem~\ref{thm.spectral} we have
\begin{equation}\label{eq.eigenvector}
   \langle\xi(x),Af\rangle_{H_{\tau}}=x\langle\xi(x),f\rangle_{H_{\tau}},
   \quad f\in l_{\rm fin},
\end{equation}
and the mapping
\begin{equation}\label{eq.self.unitary.moment}
   H_{\tau}\supset l_{\rm fin}\ni f\mapsto(I_Af)(\cdot):=\langle
   f,\xi(\cdot)\rangle_{H_{\tau}}\in L^2(\mathbb{R},\mu)
\end{equation}
is isometric.

To prove (\ref{eq.moment-reprez}), it suffices to check that
\begin{equation}\label{eq.Ip-Ia}
    (I_Af)(x)=(I_Pf)(x)=\sum_{n=0}^{\infty}f_nP_n(x),\quad f\in l_{\rm
    fin},
\end{equation}
for $\mu$-almost all $x\in\mathbb{R}$.

Indeed, suppose that (\ref{eq.Ip-Ia}) takes place. Then by
(\ref{eq.pars.equality}) we have
\begin{equation}\label{eq.911}
    (f,g)_{H_{\tau}}=\int_{\mathbb{R}}(I_Pf)(x)\overline{(I_Pg)(x)}\,d\mu(x),
    \quad f,g\in l_{\rm fin}.
\end{equation}
Therefore, taking into account the equalities
$\tau_n=\tau(\delta_n)=\tau(\delta_n\ast_P\delta_0)=(\delta_n,\delta_0)_{H_{\tau}}$
and $(I_P\delta_n)(x)=P_n(x)$, we get
$$
   \tau_n=(\delta_n,\delta_0)_{H_{\tau}}=\int_{\mathbb{R}}P_n(x)\,d\mu(x),\quad n\in \mathbb{N}_0.
$$

Let us check (\ref{eq.Ip-Ia}). According to \cite{B02a}, Lemma~2.2,
there exists a unique determined unitary operator
$U:(l^2(p))_{H_{\tau}}'\to l^2(p^{-1})$ such that
$$
   \langle U\eta,g\rangle_{l^2}=\langle \eta,g\rangle_{H_{\tau}},
   \quad \eta\in (l^2(p))_{H_{\tau}}',\quad g\in l^{2}(p).
$$
Therefore, it suffices to show that
$$
   (U\xi)(x)=P(x):=(P_n(x))_{n=0}^{\infty},
   \quad x\in\mathbb{R},
$$
or, equivalently,
$$
   \langle P(x),Af\rangle_{l^2}=x\langle P(x),f\rangle_{l^2},
   \quad x\in\mathbb{R},\quad f\in l_{\rm fin}.
$$

But the latter equality takes place, since on the one hand
$$
   x\langle
   P(x),f\rangle_{l^2}=x\sum_{n=0}^{\infty}f_nP_n(x)=x\cdot(I_Pf)(x).
$$
On the other hand, taking into account that $P_1(x)=x$,
$Af=\delta_1\ast_Pf$, $f\in l_{\rm fin}$, and
$$
    (\delta_1\ast_Pf)_n=\sum_{k=0}^{\infty}f_k(P_1P_k,P_n)_{\mathcal{P}},
    \quad n\in\mathbb{N}_0,
$$
we get
\begin{align*}
    \langle P(x),Af\rangle_{l^2}
   &
    =\langle P(x),\delta_1\ast_Pf\rangle_{l^2}
    =\sum_{n=0}^{\infty}P_n(x)\Big(\sum_{k=0}^{\infty}f_k(P_1P_k,P_n)_{\mathcal{P}}\Big)\\
   &
    =\sum_{n=0}^{\infty}\Big(\sum_{k=0}^{\infty}f_k(P_1P_k,P_n)_{\mathcal{P}}P_n(x)\Big)
    =\sum_{n=0}^{\infty}(P_1I_Pf,P_n)_{\mathcal{P}}P_n(x)\\
   &
    =P_1(x)\cdot(I_Pf)(x)=x\cdot(I_Pf)(x).
\end{align*}

Thus, Theorem~\ref{t.moment-problem} is proved .
\end{proof}

\begin{rem}\label{r.class-moment}
   Let $P_n(x)=x^n,\, n\in\mathbb{N}_0$. Then $\ast_P=\ast$ is the Cauchy product (\ref{eq.convolution.classic})
   and the corresponding moment problem is called the {\it Hamburger moment
   problem}.

   From Theorem~\ref{t.moment-problem} and formula (\ref{eq.convolution.classic})
   we immediately get the following classical result: $\tau=(\tau_n)_{n=0}^{\infty}\in\mathbb{C}^{\infty}$ {\it is a
   moment functional on $(\mathcal{A},\ast)$ (moment sequences) if and only if}
   \begin{equation}\label{eq.202-1}
      \tau(f\ast\bar{f})=\sum_{j,k=0}^{\infty}\tau_{j+k}f_j\bar{f}_k\geq
      0,\quad f\in l_{\rm fin}.
   \end{equation}

   Note that now the operator $J_P:l_{\rm fin}\to l_{\rm fin},\; J_Pf:=\delta_1\ast
   f$,
   is an ordinary right shift (or, in another terminology, a creation operator), that is
   $$
       J_Pf=J(f_0,f_1,\ldots)=(0,f_0,f_1,\ldots),
       \quad f=(f_n)_{n=0}^{\infty}\in l_{\rm fin},
   $$
   or in a matrix form
\begin{equation*}
      J_P=
      \begin{pmatrix}
        0 & 0 & 0  & 0 & 0 & \ldots \\
        1 & 0 & 0  & 0 & 0 & \ldots \\
        0 & 1 & 0  & 0 & 0 & \ldots \\
        0 & 0 & 1  & 0 & 0 & \ldots \\
        \cdot & \cdot & \cdot & \cdot & \cdot & \ldots \
     \end{pmatrix}.
\end{equation*}
\end{rem}

\begin{rem}
   Let $P_n(x)=(x)_n=x(x-1)\cdots(x-n+1)$. Then $\ast_P=\star$ has form (\ref{eq.convolution.k-k})
   and as a direct consequence of Theorem~\ref{t.moment-problem} we
   get: $\tau=(\tau_n)_{n=0}^{\infty}\in\mathbb{C}^{\infty}$ {\it is a
   moment functional on $(\mathcal{A},\star)$ if and only if}
   \begin{equation}\label{eq.202-1-newton}
      \tau(f\star\bar{f})=\sum_{i,j,k=0}^{\infty}\frac{(i+j)!(i+k)!}{i!k!j!}\tau_{i+j+k}f_{i+j}\bar{f}_{j+k}\geq
      0,\quad f\in l_{\rm fin}.
   \end{equation}

It easy to see that the Newton polynomials $(x)_n$ obey the
recurrence relation 
\begin{equation}\label{eq.multiplication-on-x}
 x(x)_n=(x)_{n+1}+n(x)_n.
\end{equation}

Indeed, let $a_-:l_{\rm fin}\to l_{\rm fin}$ be an annihilation
operator,  i.e.,
  $$
     a_-((f_n)_{n=0}^{\infty})=(f_1,2f_2,\ldots,nf_n,\ldots).
  $$
  Then on the one hand, the operator
  $\partial:=I_Pa_-I_P^{-1}:\mathbb{C}[x]\to\mathbb{C}[x]$
  acts by the formula
  $$
      \partial(x)_{n}=n(x)_{n-1},\quad n\in\mathbb{N}_0.
  $$
  On the other hand, it can be proved that, for any polynomial $F\in\mathbb{C}[x]$,
  $$
     (\partial F)(x)=F(x+1)-F(x)\quad\text{and, therefore,}\quad \partial(x)_{n}=(x+1)_n-(x)_n.
  $$
  Thus,
  $
      n(x)_{n-1}=(x+1)_n-(x)_n
  $
  and therefore (\ref{eq.multiplication-on-x}) holds.

It follows from (\ref{eq.multiplication-on-x}) that the operator
$J_P:l_{\rm fin}\to l_{\rm fin},\; J_Pf:=\delta_1\star f$, has the
following matrix representation
\begin{equation*}
      J_P=
      \begin{pmatrix}
        0 & 0 & 0  & 0 & 0 & \ldots \\
        1 & 1 & 0  & 0 & 0 & \ldots \\
        0 & 1 & 2  & 0 & 0 & \ldots \\
        0 & 0 & 1  & 3 & 0 & \ldots \\
        \cdot & \cdot & \cdot & \cdot & \cdot & \ldots \
     \end{pmatrix}.
\end{equation*}
\end{rem}





\section{Sheffer polynomials and analytic measures}\label{s.sheffer.polynom}

The {\it Sheffer polynomials} $(P_n)_{n=0}^{\infty}$ are defined via
their exponential generating function
\begin{equation}\label{eq.Schefer}
   P(x,\lambda):=\gamma(\lambda)e^{\alpha(\lambda)x}
   =\sum_{n=0}^{\infty}\frac{\lambda^n}{n!}P_n(x),
  \quad x\in{\mathbb R},\quad\lambda\in \mathcal{U},
\end{equation}
where $\mathcal{U}$ is a some neighborhood of zero in $\mathbb{C}$,
$\gamma$ and $\alpha$ are analytic functions in $\mathcal{U}$ such
that $\alpha(0)=0,\,\alpha'(0)\neq0$ and $\gamma(0)=1$. Using the
classical Faa di Bruno formula it can be showed that each $P_n(x)$
is a polynomial of exact degree $n\in\mathbb{N}_0$.

We observe that many classical polynomial families are Sheffer ---
the monomials, Newton, Bernoulli, Hermite, Poisson-Charlier
polynomials and many others. Note also that the Sheffer polynomials
have remarkable applications in various fields, such as probability,
numerical analysis, Rota's umbral calculus and so on. We refer,
e.g., to \cite{Sheffer, Rot75, Roman} and \cite{Anshelevich} for
more details.

For every $x\in \mathbb{R}$ the function $P(x,\cdot)$ is an analytic
in a neighborhood of $0\in\mathbb{C}$. Therefore,
$$
   P_n(x)=\frac{d^n}{d\lambda^n}P(x,\lambda)\Big|_{\lambda=0}
   =\frac{n!}{2\pi
   i}\oint_{|\zeta|=r}\frac{P(x,\zeta)}{\zeta^{n+1}}d\zeta,
$$
where $r>0, r\in \mathcal{U}$. As a consequence, for all
$\varepsilon>0$ there exists $r_{\varepsilon}>0$ such that
\begin{equation}\label{eq.estimate.polinom}
    |P_n(x)|\leq
    \frac{n!}{r_{\varepsilon}^n}\sup_{|\lambda|=r_{\varepsilon}}|\gamma(\lambda)e^{\alpha(\lambda)x}|
    \leq \frac{2n!}{r_{\varepsilon}^n}e^{\varepsilon|x|}, \quad x\in\mathbb{R},\quad
    n\in\mathbb{N}_0,
\end{equation}
where $r_{\varepsilon}\in \mathcal{U}$ is chosen in such a way that
$|\alpha(\lambda)|\leq\varepsilon$ and $\gamma(\lambda)\leq 2$ for
$|\lambda|=r_{\varepsilon}$.

Let $\mu$ be a non-negative finite Borel measure on $\mathbb{R}$
such that a Laplace transform
\begin{equation*}
              l_{\mu}(\lambda):=\int_{\mathbb{R}}e^{x\lambda}\,d\mu(x)
\end{equation*}
is well-defined in a neighborhood of zero in $\mathbb{C}$. It is
easy to check the following result.

\begin{prop}\label{r.analit-measure}
  Let a measure $\mu$ on $\mathcal{B}(\mathbb{R})$  be such that  $e^{x\lambda}$
  belongs to $L^1(\mathbb{R},\mu)$ for $|\lambda|<\varepsilon$ (for some $\varepsilon>0$).
  Then the Laplace transform $l_{\mu}$ of $\mu$ admits the representation
  \begin{equation}\label{eq.analytic-laplace}
              l_{\mu}(\lambda):=\int_{\mathbb{R}}e^{x\lambda}\,d\mu(x)=
              \sum_{n=0}^{\infty}\frac{\lambda^n}{n!}\int_{\mathbb{R}}x^n\,d\mu(x),
              \quad |\lambda|<\varepsilon,
  \end{equation}
  and, as a consequences, $l_{\mu}$ is analytic in a neighborhood of
  zero in $\mathbb{C}$, i.e., $l_{\mu}\in \rm{Hol}_0(\mathbb{C})$.
\end{prop}

\begin{proof}
  Let us fix $\lambda\in\mathbb{C}$ such that $|\lambda|<\varepsilon$.  Since $e^{x\lambda}\in L^1(\mathbb{R},\mu)$
  then $\cosh(x|\lambda|)\in L^1(\mathbb{R},\mu)$ and by the
  monotone convergence theorem we have
$$
   \int_{\mathbb{R}}\cosh(x|\lambda|)\,d\mu(x)=
   \sum_{n=0}^{\infty}\frac{|\lambda|^{2n}}{(2n)!}\int_{\mathbb{R}}x^{2n}\,d\mu(x)<\infty.
$$
Therefore, $x^{2n}\in L^1(\mathbb{R},\mu)$. Using the Shwarz
inequality we get
$$
   \int_{\mathbb{R}}|x|^n\,d\mu(x)
   \leq\sqrt{\mu(\mathbb{R})}\Big(\int_{\mathbb{R}}|x|^{2n}\,d\mu(x)\Big)^{\frac{1}{2}}<\infty,
$$
i.e., $x^{n}\in L^1(\mathbb{R},\mu)$ for all $n\in\mathbb{N}$. Since
$|\sum_{n=0}^{N}\frac{\lambda^n}{n!}x^n|\leq 2\cosh(x|\lambda|)$, by
the dominated convergence theorem we obtain
$$
     l_{\mu}(\lambda)=\int_{\mathbb{R}}e^{x\lambda}\,d\mu(x)=
     \sum_{n=0}^{\infty}\frac{\lambda^n}{n!}\int_{\mathbb{R}}x^n\,d\mu(x)<\infty,
     \quad |\lambda|<\varepsilon.
$$
\end{proof}

Denote by $\mathcal{M}_a(\mathbb{R})$ the set of all non-negative
finite analytic measures $\mu$ on $\mathcal{B}(\mathbb{R})$, i.e.,
$$
   \mathcal{M}_a(\mathbb{R}):=
   \big\{\mu:\mathcal{B}(\mathbb{R})\to[0,\infty)\,\big|\,\mu\,\text{-- measure},\,l_{\mu}\in
   \rm{Hol}_0(\mathbb{C})\big\}.
$$
Equivalent descriptions of analytic measures are given by the
following lemma (see \cite{KSW95} for the infinite dimensional
analogue of this result).

\begin{thm}\label{t.measure.analitic}
 The following statement are equivalent
  \begin{enumerate}
    \item $\mu\in\mathcal{M}_a(\mathbb{R})$.
    \item There exists a constant $C>0$ such that
          $$
              \Big|\int_{\mathbb{R}}x^n\,d\mu(x)\Big|<n!C^{n+1},\quad
              n\in\mathbb{N}_0.
          $$
    \item There exists a constant $r>0$ such that $e^{|x|r}\in
    L^1(\mathbb{R},\mu)$.
    \item There exists a constant $\varepsilon>0$ such that
    $P(x,\lambda)=\gamma(\lambda)e^{\alpha(\lambda)x}\in
    L^1(\mathbb{R},\mu)$ for $|\lambda|<\varepsilon$, where $P(x,\lambda)$
    is a generating function of the Sheffer polynomials $P_n(x)$.
  \end{enumerate}
\end{thm}

\begin{proof}
 Let us check the following chain $1)\Rightarrow 2) \Rightarrow 3) \Rightarrow 4)\Rightarrow 1)$.

 $1)\Rightarrow 2)$. This fact immediately follows from representation (\ref{eq.analytic-laplace}).

 $2)\Rightarrow 3)$.
 For the moments of even order we have
 $$
     \int_{\mathbb{R}}|x|^{2n}\,d\mu(x)=\int_{\mathbb{R}}x^{2n}\,d\mu(x)\leq
     (2n)!C^{2n+1},\quad n\in\mathbb{N}_0.
 $$
 The moments of arbitrary order can be estimated by the Cauchy–-Bunyakovsky–-Schwarz
 inequality
 \begin{align*}
   \int_{\mathbb{R}}|x|^n\,d\mu(x)&\leq\sqrt{\mu(\mathbb{R})}\Big(\int_{\mathbb{R}}|x|^{2n}\,d\mu(x)\Big)^{\frac{1}{2}}
   \leq\sqrt{\mu(\mathbb{R})C}
   C^n\sqrt{(2n)!}\leq\sqrt{\mu(\mathbb{R})C}(2C)^nn!,
 \end{align*}
 since $(2n)!\leq 4^n(n!)^2$. Chose $r<(2C)^{-1}$ then
 $$
    \int_{\mathbb{R}}e^{|x|r}\,d\mu(x)
    =\sum_{n=0}^{\infty}\frac{r^n}{n!}\int_{\mathbb{R}}|x|^n\,d\mu(x)
    =\sqrt{\mu(\mathbb{R})C}\sum_{n=0}^{\infty}(r2C)^n<\infty.
 $$

 $3)\Rightarrow 4)$. Let $r>0$ be such as in statement (3) and $\varepsilon>0$ be
 chosen in such way that
 $\varepsilon\in B_0$ and $|\alpha(\lambda)|\leq r$ for $|\lambda|<\varepsilon$.
 Then for all $x\in\mathbb{R}$ and all $\lambda\in \mathbb{C}$ such that $ |\lambda|<\varepsilon$ we have
 $$
     |P(x,\lambda)|=|\gamma(\lambda)e^{\alpha(\lambda)x}|\leq
     Ce^{r|x|},\quad
     C:=\sup_{|\lambda|\leq\varepsilon}|\gamma(\lambda)|.
 $$
 So, $P(x,\lambda)\in  L^1(\mathbb{R},\mu)$ for $|\lambda|<\varepsilon$.

 $4)\Rightarrow 1)$.
 Since $\alpha$ is the analytic function in $\mathcal{U}$, $\alpha(0)=0$ and $\alpha'(0)\neq0$
 then there exists $\tilde{\varepsilon}>0$ such that
 $\{\lambda\in\mathbb{C}\,|\,|\lambda|<\tilde{\varepsilon}\}\subset {\rm Ran}\,(\alpha)$.
 Therefore, $e^{x\lambda}\in L^1(\mathbb{R},\mu)$ for
 $|\lambda|<\tilde{\varepsilon}$, i.e.,
 $\mu\in\mathcal{M}_a(\mathbb{R})$.
\end{proof}

\begin{cor}\label{l.l1-l2}
  Let $\mu\in\mathcal{M}_a(\mathbb{R})$ and $P(x,\lambda):=\gamma(\lambda)e^{\alpha(\lambda)x}$ be
  a generating function of the Sheffer polynomials $P_n(x)$. Then the following formula holds
  $$
    \int_{\mathbb{R}}P(x,\lambda)\,d\mu(x)=\sum_{n=0}^{\infty}\frac{\lambda^n}{n!}\int_{\mathbb{R}}P_n(x)\,d\mu(x)<\infty,
    \quad |\lambda|<\varepsilon,
  $$
  for some $\varepsilon>0$.
\end{cor}

\begin{proof}
Clearly, $P^2(x,\lambda)=\gamma^2(\lambda)e^{2\alpha(\lambda)x}$ is
a generating function of the Sheffer polynomials. Therefore, by
Theorem~\ref{t.measure.analitic} we have $P^2(\cdot,\lambda)\in
L^1(\mathbb{R},\mu)$ for $|\lambda|<\varepsilon$, i.e.,
$P(\cdot,\lambda)\in L^2(\mathbb{R},\mu)$ for
 $|\lambda|<\varepsilon$.
Using the latter, (\ref{eq.Schefer}) and the continuity property of
the inner product, we get
\begin{align*}
    \int_{\mathbb{R}}P(x,\lambda)\,d\mu(x)
    &=
    (P(\cdot,\lambda),1)_{L^2(\mathbb{R},\mu)}
    =\sum_{n=0}^{\infty}\frac{\lambda^n}{n!}(P_n(\cdot),1)_{L^2(\mathbb{R},\mu)}\\
    &=\sum_{n=0}^{\infty}\frac{\lambda^n}{n!}\int_{\mathbb{R}}P_n(x)\,d\mu(x)<\infty,
    \quad |\lambda|<\varepsilon.
\end{align*}
\end{proof}

\begin{rem}\label{r.lapl.analytic}
   $\mu\in\mathcal{M}_a(\mathbb{R})$ if and only if there exists a constant $C>0$ such that
  \begin{equation}\label{eq.esimate.polynom}
    \int_{\mathbb{R}}|P_n(x)|^2\,d\mu(x)\leq (n!)^2C^{n+1},
    \quad n\in\mathbb{N}_0,
  \end{equation}
  where $(P_n)_{n=0}^{\infty}$ is a family of Sheffer polynomials on
  $\mathbb{R}$.
\end{rem}

\begin{proof}

Let $\mu\in\mathcal{M}_a(\mathbb{R})$. Using
(\ref{eq.estimate.polinom}) for $2\varepsilon<r$ ($r>0$ from
Theorem~\ref{t.measure.analitic}) and
Theorem~\ref{t.measure.analitic} we get
  \begin{equation*}
    \int_{\mathbb{R}}|P_n(x)|^2\,d\mu(x)\leq\frac{4(n!)^2}{r_{\varepsilon}^{2n}}\int_{\mathbb{R}}e^{2\varepsilon|x|}d\mu<\infty,
    \quad n\in\mathbb{N}_0.
  \end{equation*}
Hence, (\ref{eq.esimate.polynom}) takes place.

Conversely, let (\ref{eq.esimate.polynom}) holds. Then
$$
   \Big\|\sum_{n=M}^{N}\frac{|\lambda|^n}{n!}P_n(\cdot)\Big\|_{L^2(\mathbb{R},\mu)}\leq
   \sum_{n=M}^{N}\frac{|\lambda|^n}{n!}\|P_n(\cdot)\|_{L^2(\mathbb{R},\mu)}\leq
   \sqrt{C}\sum_{n=M}^{N}(\lambda \sqrt{C})^n.
$$
So, $P(\cdot,\lambda)\in L^2(\mathbb{R},\mu)$ for $|\lambda|<
(\sqrt{C})^{-1}$ and, therefore, from
Theorem~\ref{t.measure.analitic} follows that
$\mu\in\mathcal{M}_a(\mathbb{R})$.
\end{proof}


\section{Analytic moment functionals}\label{s.analitik-moment-problem}

\subsection{Definition and properties}

As above let $P(x,\lambda):=\gamma(\lambda)e^{\alpha(\lambda)x}$ be
a generating function of the  Sheffer polynomials $P_n(x)$ and
$\mathcal{A}=l_{\rm fin}$ be an algebra with the product $\ast_P$
(\ref{eq.convolution.g}). In the sequel, we wil fix such family
$(P_n(x))_{n=0}^{\infty}$ of Sheffer polynomials and we assume, in
addition, that $P_n(x),\,n\in\mathbb{N}_0,$ are real-valued
polynomials.

\begin{defn}
 A functional $\tau\in \mathbb{C}^{\infty}$ is
 said to be an {\it analytic moment functional on} ($\mathcal{A},\ast_P$) if there exists an
 analytic measure $\mu\in\mathcal{M}_a(\mathbb{R})$ such that
\begin{equation}\label{eq.moment-reprez-new}
   \tau_n=\int_{\mathbb{R}}P_n(x)\,d\mu(x),\quad n\in \mathbb{N}_0.
\end{equation}
\end{defn}

Clearly, $\tau\in \mathbb{C}^{\infty}$ is an analytic moment
functional on $\mathcal{A}$ if and only if $\tau$ is a moment
functional on $\mathcal{A}$ and the measure $\mu$ in representation
(\ref{eq.moment-reprez-new}) belongs to $\mathcal{M}_a(\mathbb{R})$.

\begin{rem}\label{rem.m-g}
 Let $\tau=(\tau_n)_{n=0}^{\infty}\in \mathbb{C}^{\infty}$
 be an analytic moment functional on $\mathcal{A}$ and $\mu$
 be the corresponding measure on $\mathcal{B}(\mathbb{R})$ such that (\ref{eq.moment-reprez-new}) holds. Then
 Corollary~\ref{l.l1-l2} shows that the moments $\tau_n$ are the Taylor coefficients
 of the function generalized Laplace transform
 $$
    l(\cdot):=\int_{\mathbb{R}}P(x,\cdot)\,d\mu(x)\in \rm{Hol}_0(\mathbb{C}).
 $$
 That is, $l$ is the generating function for the moments $\tau_n$  and
 $$
    \tau_n=l^{(n)}(0)=\frac{d^n}{d\lambda^n}l(\lambda)\Big|_{\lambda=0},
    \quad n\in\mathbb{N}_0.
 $$
By using the $S$-transform (see Theorem~\ref{t.s-transform}), this
means that $\tau=S^{-1}l. $
\end{rem}

The aim of this section is to find conditions on $\tau\in
\mathbb{C}^{\infty}$ that would guarantee existence of a measure
$\mu$ on $\mathcal{B}(\mathbb{R})$ such that
$\mu\in\mathcal{M}_a(\mathbb{R})$ and (\ref{eq.moment-reprez-new})
takes place.

\begin{thm}\label{t.moment.sheffer.general}
  Necessary conditions that $\tau=(\tau_n)_{n=0}^{\infty}\in \mathbb{C}^{\infty}$ is an analytic moment functional on
  $\mathcal{A}$ are the following: $\tau$ is $\ast_P$-positive on $\mathcal{A}$
  (i.e., (\ref{eq.positivity-c}) holds)  and $\tau\in  l_{-}^2$.

  Sufficient conditions that $\tau=(\tau_n)_{n=0}^{\infty}\in \mathbb{C}^{\infty}$ is an analytic moment functional on
  $\mathcal{A}$ are the following:  $\tau$ is $\ast_P$-positive on $\mathcal{A}$
  and there is a constant $C>0$ such that
   \begin{equation}\label{eq.rtfgj}
    \tau(\delta_n\ast_P\delta_n)=\sum_{k=0}^{2n}\tau_k(\delta_n\ast_P\delta_n)_k
    \leq (n!)^2C^{n+1},\quad   n\in\mathbb{N}_0,
   \end{equation}
 where the vector $\delta_n\in l_{\rm fin}$ is defined by
 (\ref{eq.delta_n}).

\end{thm}

\begin{proof}
{\it Necessity}. Let $\tau\in\mathbb{C}^{\infty}$ be an analytic
moment functional on $\mathcal{A}$. Then from
Theorem~\ref{t.moment-problem} it immediately follows that $\tau$ is
$\ast_P$-positive on $\mathcal{A}$. Since by Remark~\ref{rem.m-g}
the function
$l(\lambda)=\sum_{n=0}^{\infty}\frac{\lambda^n}{n!}\tau_n$ belongs
to the space $\rm{Hol}_0(\mathbb{C})$, the fact that $\tau\in
l_{-}^2$ is a direct consequence of Theorem~\ref{eq.s-transform}.

{\it Sufficiency}. Suppose that $\tau$ is $\ast_P$-positive on
$\mathcal{A}$ and (\ref{eq.rtfgj}) holds. Then according to
Theorem~\ref{t.moment-problem} the $\ast_P$-positiveness of $\tau$
insures that  (\ref{eq.moment-reprez-new}) holds. Next, using
(\ref{eq.911}), (\ref{eq.Ip-Ia})) and (\ref{eq.rtfgj}) we get
\begin{align*}
    \tau(\delta_n\ast_P\delta_n)
    &=(\delta_n,\delta_n)_{H_{\tau}}
    =\int_{\mathbb{R}}(I_P\delta_n)^2(x) \,d\mu(x)\\
    &=\int_{\mathbb{R}}P^2_n(x) \,d\mu(x)\leq (n!)^2C^{n+1},\quad
    n\in\mathbb{N}_0.
\end{align*}
So, from Remark~\ref{r.lapl.analytic} we conclude that
$\mu\in\mathcal{M}_a(\mathbb{R})$.
\end{proof}

\begin{thm}\label{t.exp.unique}
 If $\tau=(\tau_n)_{n=0}^{\infty}\in \mathbb{C}^{\infty}$
 is an analytic moment functional on ($\mathcal{A},\ast_P$) then the measure
 $\mu$ in representation (\ref{eq.moment-reprez-new}) is uniquely defined.
\end{thm}

\begin{proof}
At first, we prove the statement for the case $P_n(x)=x^n$. So, we
need to show that for an analytic moment functional
$\tau=(\tau_n)_{n=0}^{\infty}\in \mathbb{C}^{\infty}$ on
($\mathcal{A},\ast$) ($\ast$ is defined by (\ref{eq.202-1})) the
measure $\mu$ in the representation
\begin{equation}\label{eq.moment-reprez-prf}
   \tau_n=\int_{\mathbb{R}}x^n\,d\mu(x),\quad n\in \mathbb{N}_0,
\end{equation}
is unique.

It is well known (see, e.g., \cite{B65}, Ch.~8, Theorem~1.1) that
the measure $\mu$ in representation (\ref{eq.moment-reprez-prf}) is
unique if and only if the operator $l_{\rm fin}\ni f\mapsto
Jf=\delta_1\ast f\in l_{\rm fin}$ is essentially self-adjoint (i.e.,
has a unique self-adjoint extension) in the space $\mathcal{H}_\tau$
(see the proof of Theorem~\ref{t.moment-problem} for the definition
of $\mathcal{H}_\tau$). Since ${\rm
span}\{\delta_n\,|\,n\in\mathbb{N}_0\}=l_{\rm fin}$ and $l_{\rm
fin}$ is dense in $\mathcal{H}_\tau$, then according to the
quasianalytic criterion of self-adjointness (see
Theorem~\ref{thm.quasianalytic}) it is sufficient to check that
every vector $\delta_k, k\in\mathbb{N}_0$, is quasianalytic, i.e.,
equality (\ref{eq.quasianalytic}) holds for every $\delta_k$.

It is easy to see that $J^n\delta_k=\delta_{k+n}$,
$\|J^n\delta_k\|_{\mathcal{H}_\tau}^2=\|\delta_{k+n}\|_{\mathcal{H}_\tau}^2=\tau_{2k+2n}$.
Since $\mu\in\mathcal{M}_a(\mathbb{R})$, then there exists $C>0$
such that $|\tau_n|\leq n!C^{n+1}$ for all $n\in\mathbb{N}_0$ and
therefore
$$
   \sum_{n=1}^{\infty}\frac{1}{\sqrt[n]{\|J^n\delta_k\|_{\mathcal{H}_\tau}}}=
   \sum_{n=1}^{\infty}\frac{1}{\sqrt[2n]{\tau_{2k+2n}}}=\infty,
   \quad k\in\mathbb{N}_0,
$$
i.e., the measure $\mu$ in representation
(\ref{eq.moment-reprez-prf}) is unique.

Let us prove the general case. Suppose that measures
$\mu_1,\mu_2\in\mathcal{M}_a(\mathbb{R})$ are such that
$\mu_1\neq\mu_2$ and
\begin{equation*}
   \int_{\mathbb{R}}P_n(x)d\mu_1(x)=\int_{\mathbb{R}}P_n(x)d\mu_2(x),\quad
   n\in\mathbb{N}_0.
\end{equation*}
Then it is easy to check by induction that
\begin{equation*}
   \int_{\mathbb{R}}x^nd\mu_1(x)=\int_{\mathbb{R}}x^nd\mu_2(x),\quad
   n\in\mathbb{N}_0.
\end{equation*}
So, $\mu_1=\mu_2$, due to the above proven, which leads to a
contradiction.
\end{proof}

\begin{rem}
From the proof of Theorem~\ref{t.exp.unique} and
Remark~\ref{r.Fouirier} it easily follows the next well known
result: {\it If $\mu\in\mathcal{M}_a(\mathbb{R})$ then the set of
all polynomials $\mathbb{C}[x]$  is dense in the space
$L^2(\mathbb{R},\mu)$}.
\end{rem}



\subsection{Analytic moment functionals connected with the monomials}\label{s.moment-problem-exp-convex}

Let $P(x,\lambda)=e^{x\lambda}$ and $\mathcal{A}=l_{\rm fin}$ be an
algebra with the Cauchy product $\ast$
(\ref{eq.convolution.classic}).


\begin{thm}\label{t.exp.kriterij}
 A functional $\tau\in \mathbb{C}^{\infty}$
 is  an analytic moment functional on $(\mathcal{A},\ast)$, i.e.,
there exists a measure $\mu\in\mathcal{M}_a(\mathbb{R})$ such that
\begin{equation}\label{eq.moment-class}
   \tau(\delta_n)=\tau_n=\int_{\mathbb{R}}x^n\,d\mu(x),\quad n\in \mathbb{N}_0,
\end{equation}
 if and only if $\tau$ is $\ast$-positive on $\mathcal{A}$ (i.e., (\ref{eq.202-1}) holds)  and $\tau\in  l_{-}^2$.

 For an analytic moment functional $\tau\in \mathbb{C}^{\infty}$ on $(\mathcal{A},\ast)$ the
 measure $\mu$ in representation (\ref{eq.moment-class}) is unique defined.
\end{thm}

\begin{proof}
The necessity immediately follows from
Theorem~\ref{t.moment.sheffer.general}.

Let us prove the sufficiency. Assume that $\tau$ is $\ast$-positive
and $\tau\in l_{-}^2$. Then from Remark~\ref{r.class-moment} we
conclude that $\tau$ is a moment functional on $(\mathcal{A},\ast)$,
i.e., there exists a Borel measure $\mu$ on $\mathbb{R}$ such that
(\ref{eq.moment-class}) holds.

Let us check that $\mu\in\mathcal{M}_a(\mathbb{R})$. Since
$\tau=(\tau_n)_{n=0}^{\infty}\in l_{-}^2$, from
Corollary~\ref{cor.hol} it follows that there exists $C>0$ such that
\begin{equation*}
   |\tau_n|=\Big|\int_{\mathbb{R}}x^n\,d\mu(x)\Big|\leq n!C^{n+1},\quad
   n\in\mathbb{N}_0.
\end{equation*}
Hence, from Theorem~\ref{t.measure.analitic} we conclude that
$\mu\in\mathcal{M}_a(\mathbb{R})$.

The last assertion of the theorem directly follows from
Theorem~\ref{t.exp.unique}.
\end{proof}

Let us show that the class of analytic moment functionals on
$(\mathcal{A},\ast)$ is closely related to the class of
exponentially convex functions. Recall that {\it a function
$k:(-2a,2a)\to\mathbb{C}$, where $0<a\leq\infty$, is called
exponentially convex if
  \begin{equation}\label{eq.exp.convex}
     \sum_{i,j=0}^{\infty}k(x_i+x_j)f_i\bar{f}_j\geq 0
  \end{equation}
  for all $f=(f_n)_{n=0}^{\infty}\in l_{\rm fin}$ and $x_i, x_j\in (-a,a)$.
}

The classical Bernstein's theorem asserts (see, e.g.,
\cite{Akhiezer}, Ch. 5, \S~5; \cite{B65}, Ch. 8, \S~3): {\it A
continuous function $k:(-2a,2a)\to\mathbb{C}$ is exponentially
convex if and only if there exists a non-negative finite Borel
measure $\mu$ on $\mathbb{R}$ such that
 $$
     k(\lambda)=l_{\mu}(\lambda)=\int_{\mathbb{R}}e^{x\lambda}\,d\mu(x),
     \quad \lambda\in (-2a,2a).
 $$
 The measure $\mu$ in the latter representation is unique}.
 It follows from Theorem~\ref{t.measure.analitic} that in fact
 $\mu\in\mathcal{M}_a(\mathbb{R})$ and therefore $k:(-2a,2a)\to\mathbb{C}$ is an analytic in a
 neighborhood of zero in $\mathbb{R}$.

From Bernstein's theorem, Theorem~\ref{t.exp.kriterij} and
Remark~\ref{rem.m-g} (for $P(x,\lambda)=e^{x\lambda}$) we get the
following result.

\begin{thm}\label{t.exp-conv-analyt}
 A functional $\tau=(\tau_n)_{n=0}^{\infty}\in \mathbb{C}^{\infty}$
 is an analytic moment functional on $(\mathcal{A},\ast)$
 if and only if the function 
\begin{equation*}
    k(\lambda):=\sum_{n=0}^{\infty}\frac
    {\lambda^n}{n!}\tau_n
\end{equation*}
is well defined and exponentially convex  in some neighborhood of
zero in $\mathbb{R}$.

Vice versa, an analytic in a neighborhood $\mathcal{U}$ of zero in
$\mathbb{R}$ function $k:\mathcal{U}\to\mathbb{C}$ is exponentially
convex if and only if the functional
$$
    \tau=(\tau_n)_{n=0}^{\infty}=(k^{(n)}(0))_{n=0}^{\infty},\quad
    \tau_n:=k^{(n)}(0)=\frac{d^n}{d\lambda^n}k(\lambda)\Big|_{\lambda=0},
$$
is an analytic moment functional on $(\mathcal{A},\ast)$.
\end{thm}

\begin{cor}
An analytic in a neighborhood $\mathcal{U}$ of zero in $\mathbb{R}$
function $k:\mathcal{U}\to\mathbb{C}$ is exponentially convex if and
only if
  \begin{equation*}
      \sum_{i,j=0}^{\infty}k^{(i+j)}(0)f_i\bar{f}_j\geq
      0,\quad f=(f_n)_{n=0}^{\infty}\in l_{\rm fin}.
   \end{equation*}
\end{cor}


\subsection{Analytic moment functionals connected with the Newton
polynomials}\label{s.newton-bogolubov}

Let $P(x,\lambda)$ be a generating function of the Newton
polynomials $(x)_n=\prod_{i=0}^{n-1}(x-i)$, that is
$$
    P(x,\lambda):=(1+\lambda)^x=e^{x\log(1+\lambda)}=\sum_{n=0}^{\infty}\frac{\lambda^n}{n!}(x)_n,\quad
     |\lambda|<1,
$$
and $\mathcal{A}=l_{\rm fin}$ be an algebra with the product $\star$
(\ref{eq.convolution.k-k}).

Now an analogue of Theorem~\ref{t.exp.kriterij} holds.

\begin{thm}\label{t.exp.kriterij.Newton}
 A functional $\tau\in \mathbb{C}^{\infty}$
 is  an analytic moment functional on $(\mathcal{A},\star)$, i.e.,
 there exists a measure $\mu\in\mathcal{M}_a(\mathbb{R})$ such that
\begin{equation}\label{eq.moment-k-k}
   \tau(\delta_n)=\tau_n=\int_{\mathbb{R}}(x)_n\,d\mu(x),\quad n\in \mathbb{N}_0,
\end{equation}
 if and only if $\tau$ is $\star$-positive on $\mathcal{A}$ (i.e., (\ref{eq.202-1-newton}) holds)  and $\tau\in  l_{-}^2$.

  For an analytic moment functional $\tau\in \mathbb{C}^{\infty}$ on $(\mathcal{A},\star)$ the
  measure $\mu$ in  (\ref{eq.moment-k-k}) is unique.
\end{thm}

\begin{proof} The necessity immediately follows from Theorem~\ref{t.moment.sheffer.general}.

Let us prove the converse.  Suppose that $\tau$ is $\star$-positive
on $\mathcal{A}$ and $\tau\in l_{-}^2$. Then due to
Theorem~\ref{t.moment.sheffer.general}  it is sufficient to show
that there exists $C>0$ such that
  \begin{equation*}
    \tau(\delta_n\star\delta_n)
    =\int_{\mathbb{R}}(x)_n^2\,d\mu(x) \leq (n!)^2C^{n+1}, \quad
    n\in\mathbb{N}_0.
  \end{equation*}

Since $\tau=(\tau_n)_{n=0}^{\infty}\in l_{-}^2$, there exists
$\widetilde{C}>0$ such that $|\tau_n|\leq n!\widetilde{C}^{n+1}$ for
all $n\in\mathbb{N}_0$. Hence, taking into account that (see
(\ref{eq.convolution.k-k}))
    \begin{equation*}
        (\delta_n\star\delta_n)_m=
        \begin{cases}
             \displaystyle\frac{(n!)^2}{((m-n)!)^2(2n-m)!}, &  \text{if}\,\,m\in\{n,\ldots,2n\},\\
             0, & \hbox{otherwise,}
        \end{cases}
    \end{equation*}
we get
\begin{equation}\label{eq.pr.Newton.1}
   \tau(\delta_n\star\delta_n)
   =
   \sum_{m=n}^{2n}\tau_m\frac{(n!)^2}{((m-n)!)^2(2n-m)!} 
    \leq \sum_{m=n}^{2n}\widetilde{C}^{m+1}\frac{m!(n!)^2}{((m-n)!)^2(2n-m)!}.
\end{equation}

Let us estimate the expression
$$
   \frac{m!(n!)^2}{((m-n)!)^2(2n-m)!}.
$$
Using the bound for the binomial coefficients
$$
   \frac{m!}{n!(m-n)!}\leq 2^m,\quad m\in\mathbb{N}_0,
$$
we get
\begin{equation}\label{eq.pr.Newton.2}
\begin{split}
    \frac{m!(n!)^2}{((m-n)!)^2(2n-m)!}
    &=\frac{(m!)^2(n!)^4}{((m-n)!)^2(n!)^2m!(2n-m)!}\\
    &\leq 4^m\frac{(n!)^4}{m!(2n-m)!}
    \leq 4^{m}(n!)^2
\end{split}
\end{equation}
for all $m\in\{n,\ldots,2n\}$.

From (\ref{eq.pr.Newton.1}) and (\ref{eq.pr.Newton.2}) we conclude
that
$$
    \tau(\delta_n\star\delta_n)
    =\int_{\mathbb{R}}(x)_n^2\,d\mu(x)
    \leq \widetilde{C}^{2n+1}4^{2n+1}(n!)^2\leq (n!)^2C^{n+1},
$$
where $C:=\max\{8\widetilde{C}^2,\,4\widetilde{C}\}$. So, the
sufficiency is proved.

The last assertion of the theorem directly follows from
Theorem~\ref{t.exp.unique}.
\end{proof}

Let us establish a relation between the analytic moment functional
on $(\mathcal{A},\star)$ and a one-dimensional analog of the
Bogoliubov generating functionals. We say that {\it a function
$B:\mathcal{U}\to\mathbb{C}$ ($\mathcal{U}$ is a neighborhood of
zero in $\mathbb{C}$) is a Bogoliubov  functional in $\mathcal{U}$
if $B$ admits the following integral representation
\begin{equation}\label{eq.Bogoliubov-functional}
    B(\lambda)=\int_{\mathbb{R}}(1+\lambda)^xd\mu(x)=\int_{\mathbb{R}}e^{x\log(1+\lambda)}d\mu(x),
    \quad \lambda\in \mathcal{U},
\end{equation}
with some non-negative finite Borel measure $\mu$ on $\mathbb{R}$.}
It follows from Theorem~\ref{t.measure.analitic} that the measure
$\mu$ in representation (\ref{eq.Bogoliubov-functional}) is actually
analytical, i.e., $\mu\in\mathcal{M}_a(\mathbb{R})$.

It should be noticed that the classical Bogoliubov or generating
functionals were introduced by N.~N.~Bogoliubov in \cite{Bogoliubov}
to define correlation functions for statistical mechanics systems
(this functional is defined by analogue with
(\ref{eq.Bogoliubov-functional}) but for measures on the space of
finite configuration). We refer to, e.g., \cite{Nazin, KKO06} for
details, historical remarks and references therein.

An analogue of Theorem~\ref{t.exp-conv-analyt} holds.
\begin{thm}\label{t.bogolub.one.dim}
 A functional $\tau=(\tau_n)_{n=0}^{\infty}\in \mathbb{C}^{\infty}$
 is an analytic moment functional on $(\mathcal{A},\star)$
 if and only if the function 
\begin{equation*}
    B(\lambda):=\sum_{n=0}^{\infty}\frac
    {\lambda^n}{n!}\tau_n
\end{equation*}
is the Bogoliubov functional in some neighborhood of zero in
$\mathbb{C}$.

Vice versa, an analytic in a neighborhood $\mathcal{U}$ of zero in
$\mathbb{C}$ function $B:\mathcal{U}\to\mathbb{C}$ is the Bogoliubov
functional in $\mathcal{U}$ if and only if the functional
$$
    \tau=(\tau_n)_{n=0}^{\infty}=(B^{(n)}(0))_{n=0}^{\infty},\quad
    \tau_n:=B^{(n)}(0)=\frac{d^n}{d\lambda^n}B(\lambda)\Big|_{\lambda=0},
$$
is an analytic moment functional on $(\mathcal{A},\star)$.
\end{thm}

\begin{cor}
An analytic in a neighborhood $\mathcal{U}$ of zero in $\mathbb{C}$
function $B:\mathcal{U}\to\mathbb{C}$ is the Bogoliubov functional
in $\mathcal{U}$ if and only if
  \begin{equation*}
      \sum_{i,j,k=0}^{\infty}\frac{(i+j)!(i+k)!}{i!k!j!}B^{(i+j+k)}(0)f_{i+j}\bar{f}_{j+k}\geq
      0,\quad f=(f_n)_{n=0}^{\infty}\in l_{\rm fin}.
   \end{equation*}
\end{cor}


\section{Infinite dimensional case}\label{s.infinite-dim-case}

The theory outlined in previous sections has an essential
development to the case of functions of infinite many variables, see
e.g. \cite{BK88, BKKL99, B02a, B03, BM07} for details. Without going
into details we present here
a few examples of the results and open problems.


\subsection{Infinite dimensional power moment problem}

Let ${\mathcal F}(H)$ be a {\it symmetric Fock space} over a real
separable Hilbert space $H$, that is
\begin{equation*}
   {\mathcal F}(H)
   :={\mathbb C}\oplus\bigoplus_{n=1}^{\infty}
   H_{\mathbb C}^{\odot n},
\end{equation*}
where ${\odot}$ stands for the symmetric tensor product ($\otimes$
is the ordinary tensor product), the subindex $\mathbb{C}$ denotes
the complexification of a real space. Thus, ${\mathcal F}(H)$ is a
complex Hilbert space of sequences $f=(f_n)_{n=0}^{\infty}$ such
that  $f_n\in H_{\mathbb C}^{{\odot}n}$ ($H_{\mathbb
C}^{{\odot}0}:=\mathbb{C}$) and
$$
   \|f\|_{{\mathcal F}(H)}^{2}=|f_0|^2+\sum_{n=1}^{\infty}
   \|f_n\|_{H_{\mathbb C}^{\odot
   n}}^{2}<\infty.
$$
For simplicity, in the sequel we will suppose that
$H=L^2(\mathbb{R}):=L^2(\mathbb{R},dt)$ and one will always
identify, in the usual way, the space $L_{\mathbb
C}^2(\mathbb{R})^{\odot n}$ with the space $L_{{\mathbb C},\,{\rm
sym}}^{2}(\mathbb{R}^n)$ of all symmetric functions from $L_{\mathbb
C}^{2}(\mathbb{R}^n)$.

Let us construct a convenient for us rigging of the Fock space
$\mathcal{F}(L^2(\mathbb{R}))$. To this end, we start with the
classical rigging
\begin{equation}\label{eq.rigging.classic}
\mathcal{D}'\supset  L^2(\mathbb{R}) \supset \mathcal{D},
\end{equation}
where ${\mathcal D}={\mathcal D}(\mathbb{R})$ is the Schwartz space
of  infinite differentiable functions on $\mathbb{R}$ with compact
supports, ${\mathcal D}'={\mathcal D}'(\mathbb{R})$ is the Schwartz
space of distributions dual of ${\mathcal D}$ with respect to the
zero space $L^2({\mathbb R})$. We denote by
$\langle\cdot\,,\cdot\rangle$ the dual pairing between elements of
${\mathcal D}'$ and ${\mathcal D}$. We preserve the notation
$\langle\cdot\,,\cdot\rangle$ for the dual pairings in tensor powers
and complexifications of chain \eqref{eq.rigging.classic}.

Using (\ref{eq.rigging.classic}) we construct the rigging
\begin{equation}\label{eq.riggin.fock}
   \mathcal{F}_{\rm{fin}}'(\mathcal{D}) \supset  \mathcal{F}(L^2(\mathbb{R}))\supset
   \mathcal{F}_{\rm{fin}}(\mathcal{D}),
\end{equation}
where $\mathcal{F}_{\rm{fin}}(\mathcal{D})$ is a space of all finite
sequences $f=(f_n)_{n=0}^{\infty}$, $f_n\in {\mathcal D}_{\mathbb
C}^{\odot n}$ (i.e., $f_n=0$ for all $n\geq$ some
$N\in\mathbb{N}_0$), $\mathcal{F}_{\rm{fin}}'(\mathcal{D})
=\times_{n=0}^{\infty}({\mathcal D}_{\mathbb C}')^{\odot n}$ is the
dual of $\mathcal{F}_{\rm{fin}}(\mathcal{D})$ with respect to
$\mathcal{F}(L^2(\mathbb{R}))$ (it consists of all sequences of the
form $(\xi_n)_{n=0}^{\infty}$, $\xi_n\in ({\mathcal D}_{\mathbb
C}')^{\odot n}$). Note that in our case  the role of the spaces
$\mathcal{F}_{\rm{fin}}'(\mathcal{D}),\,
\mathcal{F}(L^2(\mathbb{R}))$ and
$\mathcal{F}_{\rm{fin}}(\mathcal{D})$ are the same as the role of
the spaces $\mathbb{C}^{\infty},\,  l^2$ and $l_{\rm{fin}}$ in the
one-dimensional case.

Denote by $\mathcal{P}(\mathcal{D}')$ the space of all continuous
polynomials on $\mathcal{D}'$,
$$
   \mathcal{P}(\mathcal{D}'):=
   \Big\{F:\mathcal{D}'\to\mathbb{C}\,\Big|\,\exists (f_n)_{n=0}^{\infty}\in
   \mathcal{F}_{\rm{fin}}(\mathcal{D})\,:\, F(x)=\sum_{n=0}^{\infty}\langle x^{\otimes n},f_n\rangle,\,
   x\in\mathcal{D}'\Big\}.
$$
By analogy with the one-dimensional case (see
(\ref{eq.convolution.g})), using the bijection
\begin{equation*}
     I:\mathcal{F}_{\rm{fin}}(\mathcal{D})\to \mathcal{P}(\mathcal{D}'),\quad
    f=(f_n)_{n=0}^{\infty}\mapsto
    (If)(x):=\sum_{n=0}^{\infty}\langle x^{\otimes n},f_n\rangle,
\end{equation*}
we introduce a product $\ast$ on
$\mathcal{F}_{\rm{fin}}(\mathcal{D})$ by setting
\begin{equation}\label{eq.conv-fock}
    f\ast g:=I^{-1}(If\cdot Ig),\quad f,g\in\mathcal{F}_{\rm{fin}}(\mathcal{D}).
\end{equation}
It is easy to check that (cf. (\ref{eq.convolution.classic}))
\begin{equation*}
    (f\ast g)_n=\sum_{i+j=n}f_{i}\odot g_{j}=\sum_{k=0}^{\infty}f_{k}\odot g_{n-k},
    \quad f,g\in \mathcal{F}_{\rm{fin}}(\mathcal{D}).
\end{equation*}
So, $\mathcal{F}_{\rm{fin}}(\mathcal{D})$ becomes a commutative
algebra $\mathcal{A}(\mathcal{D})$ with the product $\ast$, unity
$\delta_0=\{1,0,0,\ldots\}$ and the natural involution
$f\mapsto\bar{f}$ inducted by usual complex conjugation.

Let us pass to the infinite dimensional power moment problem.

\begin{defn}
We say that $\tau=(\tau_n)_{n=0}^{\infty}\in
\mathcal{F}_{\rm{fin}}'(\mathcal{D})=\times_{n=0}^{\infty}({\mathcal
D}_{\mathbb C}')^{\odot n}$ is a {\it moment functional  on}
($\mathcal{A}(\mathcal{D}), \ast$) if there exists a finite Borel
measure $\mu$ on $\mathcal{D}'$ such that
\begin{equation}\label{eq.moment-reprez-infinite}
   \tau_n=\int_{\mathcal{D}'}x^{\otimes n}d\mu(x),\quad\text{ i.e.,}\quad
   \langle\tau_n,\cdot\rangle=\int_{\mathcal{D}'}\langle x^{\otimes n},\cdot\rangle \,d\mu(x),\quad n\in \mathbb{N}_0.
\end{equation}
\end{defn}

Before stating the result note that the Schwartz space $\mathcal{D}$
can be interpreted as a projective limit of some Sobolev spaces
$D_\sigma,\,\sigma\in \Sigma$, i.e.,
$
   \mathcal{D}={\mathop{\rm {pr\ lim}}}_{\sigma\in \Sigma}D_\sigma,
$ where $\Sigma$ denotes some set of indexes, see e.g. \cite{BK88,
BUSH} for more details.

The following statement follows from \cite{BK88} (see also
\cite{B02a}).

\begin{thm}\label{t.inf-dim-power}
 Let $\tau=(\tau_n)_{n=0}^{\infty}\in\mathcal{F}_{\rm{fin}}'(\mathcal{D})$ and
 the following two conditions are fulfilled:
 \begin{enumerate}
   \item $\tau$ is $\ast$-positive (more exactly, non-negative) on
   $\mathcal{A}(\mathcal{D})=\mathcal{F}_{\rm{fin}}(\mathcal{D})$, that is
   \begin{equation}\label{eq.positivity-c-fock}
      \tau(f\ast\bar{f})=\sum_{j,k=0}^{\infty}\langle\tau_{j+k},f_j\odot\bar{f}_k\rangle\geq
      0,\quad f\in\mathcal{A}(\mathcal{D}).
   \end{equation}
   \item there exists an index $\sigma=\sigma(\tau)\in \Sigma$ such that
   $\tau_n\in D_{-\sigma,\mathbb{C}}^{\odot n}=(D_{\sigma,\mathbb{C}}^{\odot n})'$
   for all $n\in\mathbb{N}$ and the class
   \begin{equation}\label{eq.add-cond}
      C\{s_n\},\quad s_n=\sqrt{\|\tau_{2n}\|_{D_{-\sigma,\mathbb{C}}^{\odot
      2n}}},
   \end{equation}
   is quasianalytic (for example, $s_n=n!$).
 \end{enumerate}
 Then $\tau$ is a moment functional on ($\mathcal{A}(\mathcal{D}),
 \ast$) and the measure $\mu$ in representation (\ref{eq.moment-reprez-infinite})
 is uniquely defined.

 Conversely, for every moment functional
 $\tau$ ($\mathcal{A}(\mathcal{D}),
 \ast$) conditions
 (\ref{eq.positivity-c-fock}) is fulfilled.
\end{thm}

The proof of this result is analogous to that of
Theorem~\ref{t.moment-problem}. Namely, just as in the case of the
one-dimensional moment problem, Theorem~\ref{t.inf-dim-power} is a
result of the application of the projection spectral theorem to the
family $(J(\varphi))_{\varphi\in\mathcal{D}}$ of ``creation''
operators
\begin{equation}\label{eq.field-power}
   J(\varphi):\mathcal{A}(\mathcal{D})=\mathcal{F}_{\rm{fin}}({\mathcal D})\to
   \mathcal{A}(\mathcal{D}),\quad
   J(\varphi)f:=(0,\varphi,0,0,\ldots)\ast f,
\end{equation}
acting in a Hilbert space $H_{\tau}$  associated with the
quasiscalar product
\begin{equation}\label{eq.inner-power}
(f,g)_{H_{\tau}}=\tau(f\ast \bar{g}),\quad
f,g\in\mathcal{F}_{\rm{fin}}(\mathcal{D}).
\end{equation}

Note also that, unlike the one-dimensional case, only conditions
(\ref{eq.positivity-c-fock}) it is not sufficient for existence of
representation (\ref{eq.moment-reprez-infinite}). This is connected
with impossibility, in general, to extend a family of commuting
Hermitian operators to some family of strongly commuting selfadjoint
operators. Condition (\ref{eq.add-cond}) implies that the
corresponding Hermitian operators are essential selfadjoint and
strongly commuting (therefore, the measure $\mu$ in
Theorem~\ref{t.inf-dim-power} is unique). It is possible to give a
condition weaker than (\ref{eq.add-cond}) which makes it possible to
extend these Hermitian operators to selfadjoint commuting operators,
in this case the measure $\mu$ is not unique. For the corresponding
result, see \cite{BK88}, Ch.~5,  \S~2.

\begin{rem}\label{r.analit.laplas.inf}
Using results from \cite{KSW95}, it can be shown that the infinite
dimensional analog of Theorem~\ref{t.exp.kriterij} holds. Namely,
{\it a functional $\tau\in \mathcal{F}_{\rm{fin}}'(\mathcal{D})$
admits representation (\ref{eq.moment-reprez-infinite}) with the
analytic measure $\mu\in\mathcal{M}_a(\mathcal{D}')$ (i.e.,
$\int_{\mathcal{D}'}\exp\langle x,\lambda\rangle \,d\mu(x)<\infty$
for all $\lambda$ from some neighborhood of $0\in
\mathcal{D}_{\mathbb{C}}'$) if and only if $\tau$ is $\ast$-positive
on $\mathcal{A}(\mathcal{D}')$ (i.e., (\ref{eq.positivity-c-fock})
holds)  and $\tau$ belongs to the space ${\mathcal F}_-$}. Here
${\mathcal F}_-$ is defined (similar to $l_-^2$) by the formula
$$
    {\mathcal F}_-:=\mathop{\rm ind\,lim}_{\sigma\in \Sigma, q\in{\mathbb N}}{\mathcal
    F}(-\sigma,-q)\subset\mathcal{F}_{\rm{fin}}'(\mathcal{D}),
 $$
where ${\mathcal F}(-\sigma,-q)$ is the so-called Kondratiev-type
Fock space,
\begin{equation*}
            {\mathcal F}(-\sigma,-q):
            =\Big\{f=(f_n)_{n=0}^{\infty}\in\mathcal{F}_{\rm{fin}}'(\mathcal{D})\,\Big|\,
            \|f\|_{{\mathcal F}(-\sigma,-q)}^2
            =\sum_{n=0}^{\infty}\|f_n\|_{D_{-\sigma,\mathbb{C}}^{\odot n}}^{2}
            (n!)^{-2}2^{-qn}<\infty\Big\}.
\end{equation*}
\end{rem}


\subsection{Moment problem associated with correlation
functions}\label{s-s.correlations-functions}

In this subsection it is convenient for us to interpret the Fock
space $\mathcal{F}(L^2(\mathbb{R}))$ as the space of functions on
the space of finite configurations on $\mathbb{R}$. Namely, denote
by $\Gamma^{(n)}$ the space of $n$-point configuration, i.e.,
$$
   \Gamma^{(n)}:=\{\eta\subset\mathbb{R}\,|\,|\eta|=n\},
$$
where $|\cdot|$ means cardinality of a set. As a set, $\Gamma^{(n)}$
coincides with the symmetrization of
$$
    \widehat{\mathbb{R}}^n:=\{(t_1,\ldots,t_n)\in\mathbb{R}^n\,|\,
    t_n\neq t_j\,\,\text{if}\,\, k\neq j\}.
$$
Hence, $\Gamma^{(n)}$ inherits the topology of $\mathbb{R}^n$.
Denote by $\mathcal{B}(\Gamma^{(n)})$ the corresponding Borel
$\sigma$-algebra on $\Gamma^{(n)}$ and introduce a measure $m^{(n)}$
on $\mathcal{B}(\Gamma^{(n)})$ as the image of product $m^{\otimes
n}$ of Lebesque measures $dm(t)=dt$ on $\mathcal{B}(\mathbb{R})$. It
is clear that
$$
    L^2(\Gamma^{(n)},m^{(n)})=L_{{\mathbb C},\,{\rm
sym}}^{2}(\mathbb{R}^n,m^{\otimes n}).
$$

The space $\Gamma_0$ of (all) finite configuration is defined as the
topological disjoint union
$$
   \Gamma_0=\bigsqcup_{n=0}^{\infty}\Gamma^{(n)}.
$$
Denote by $\nu$ the  {\it Lebesque-Poisson measure} on the Borel
$\sigma$-algebra $\mathcal{B}(\Gamma_0)$,
$$
    \nu:=\sum_{n=0}^{\infty}\frac{1}{n!}m^{(n)},\quad
    m^{(0)}(\varnothing):=1,
$$
and by $L^2(\Gamma_0,\nu)$ the corresponding $L^2$-space. Clearly,
the Fock space $\mathcal{F}(L^2(\mathbb{R}))$ can be identified with
the space $L^2(\Gamma_0,\nu)$ via
$$
   \mathcal{F}(L^2(\mathbb{R}))\ni (f_n)_{n=0}^{\infty}\sim
  \sum_{n=0}^{\infty}F_n(\cdot)\in L^2(\Gamma_0,\nu),
$$
where $F_0(\varnothing):=f_0$ and
$$
   F_n(\eta):=
\left\{
  \begin{array}{ll}
    f_n(t_1,\ldots,t_n),  & \text{if}\,\,\eta=(t_1,\ldots,t_n)\in \Gamma^{(n)}\\
    0, & \text{otherwise}
  \end{array},
\right.
$$
for all $n\in\mathbb{N}$. So,
$$
   \mathcal{F}(L^2(\mathbb{R}))\cong L^{2}(\Gamma_0,\nu)=\bigoplus_{n=0}^{\infty}
   L^2(\Gamma^{(n)},m^{(n)})\,.
$$
In what follows
we won't distinguish between a vector $f=(f_n)_{n=0}^{\infty}$ from
the Fock space $\mathcal{F}(L^2(\mathbb{R}))$ (and from
$\mathcal{F}_{\rm{fin}}(\mathcal{D})$) and the corresponding
function $f(\eta),\, \eta\in\Gamma_0$, i.e.,
$$
   \mathcal{F}(L^2(\mathbb{R}))\ni f=(f_n)_{n=0}^{\infty}=f(\eta),
   \quad \eta\in\Gamma_0,\quad f_n=f\upharpoonright \Gamma^{(n)}.
$$

We will need also the space $\Gamma$ of infinite configurations on
$\mathbb{R}$, i.e., the space  of all locally finite subsets  in
$\mathbb{R}$
$$
   \Gamma:=\{\gamma\subset\mathbb{R}\;|\;|\gamma\cap\Lambda|<\infty
   \;\text{for all compact}\; \Lambda\subset\mathbb{R}\}.
$$
We consider the $\sigma$-algebra $\mathcal{B}(\Gamma)$ as the
smallest $\sigma$-algebra for which all the mappings
$N_{\Lambda}:\Gamma\to\mathbb{N}_0$,
$N_{\Lambda}(\gamma):=|\gamma\cap\Lambda|$, are measurable for all
bounded Borel set $\Lambda\subset\mathbb{R}$. Note that each element
$\gamma\in\Gamma$ can be identified with a generalized function:
$$
   \Gamma\ni\gamma\mapsto \sum_{t\in\gamma}\delta_t\in\mathcal{D}',
$$
where $\delta_t$ denotes the delta function (Dirac measure) at $t$.
In this way, the space $\Gamma$ is embedded in the Schwartz space of
distributions $\mathcal{D}'$.

Let us pass to a definition of the so-called {\it Kondratiev--Kuna
convolution} $\star$. This convolution acts in
$\mathcal{F}_{\rm{fin}}(\mathcal{D})$ and we define it by analogy
with (\ref{eq.conv-fock}) but using instead of the monomials
$\langle x^{\otimes n},f_n\rangle$ an infinite dimensional analog of
the Newton polynomials.

Recall that infinite dimensional Newton polynomials
$\chi_n(x)\in({\mathcal D}')^{\odot n}$ are defined as coefficients
of the following  expansion (cf. (\ref{eq.newton-def-onedim}))
$$
  e^{\langle x, \log (1+\lambda)\rangle}
  =\sum_{n=0}^{\infty}\frac{1}{n!}\langle \chi_n(x),\lambda^{\otimes
  n}\rangle,\quad
  x\in\mathcal{D}',\quad\lambda\in\mathcal{D}_{\mathbb{C}}.
$$
It is well known that (see e.g. \cite{BT04})
$$
   \langle\chi_{n}(x),\lambda^{\otimes n}\rangle
   =\sum_{m=0}^{n-1}(-1)^{n-m-1}\frac{(n-1)!}{m!}
   \langle\lambda^{n-m},x\rangle
   \langle\chi_{m}(x),\lambda^{\otimes m}\rangle,\quad
   n\in\mathbb{N}_0,
$$
and the mapping
\begin{equation}\label{eq.bijection-kk}
     I_{\chi}:\mathcal{F}_{\rm{fin}}(\mathcal{D})\to \mathcal{P}(\mathcal{D}'),\quad
    f=(f_n)_{n=0}^{\infty}\mapsto
    (I_{\chi}f)(x):=\sum_{n=0}^{\infty}\langle \chi_n(x),f_n\rangle
\end{equation}
is bijection. Therefore, we can introduce the convolution $\star$ on
$\mathcal{F}_{\rm{fin}}(\mathcal{D})$ by setting
\begin{equation*}
    f\star g:=I_{\chi}^{-1}(I_{\chi}f\cdot I_{\chi}g),\quad f,g\in\mathcal{F}_{\rm{fin}}(\mathcal{D}).
\end{equation*}
It follows from e.g. \cite{KK02, BM07} that
\begin{equation*}
    (f\star g)(\eta)=\sum_{\eta_1\sqcup\eta_2\sqcup\eta_3=\eta}f(\eta_1\cup\eta_2)g(\eta_2\cup\eta_3),
    \quad \eta\in\Gamma_0,\quad f,g\in \mathcal{F}_{\rm{fin}}({\mathcal D}),
\end{equation*}
where the summation is taken over all partitions of $\eta$ in three
parts (parts may be empty). Note that (\ref{eq.convolution.k-k}) is
a one-dimensional analog of the latter formula.

So, the space $\mathcal{F}_{\rm{fin}}(\mathcal{D})$ endowed with the
product $\star$ becomes a commutative algebra
$\mathcal{A}(\mathcal{D})$ with unity $\delta_0=\{1,0,0,\ldots\}$
and the natural involution $f\mapsto\bar{f}$ inducted by usual
complex conjugation.

Let us pass to the corresponding infinite dimensional moment
problem.

\begin{defn}
We say that $\tau=(\tau_n)_{n=0}^{\infty}\in
\mathcal{F}_{\rm{fin}}'({\mathcal
D})=\times_{n=0}^{\infty}({\mathcal D}_{\mathbb C}')^{\odot n}$ is a
{\it moment functional  on} ($\mathcal{A}(\mathcal{D}), \star$) if
there exists a finite Borel measure $\mu$ on $\mathcal{D}'$ such
that
\begin{equation}\label{eq.moment-reprez-infinite-kk}
   \tau_n=\int_{\mathcal{D}'}\chi_n(x)\,d\mu(x),\quad\text{ i.e.,}\quad
   \langle\tau_n,\cdot\rangle=\int_{\mathcal{D}'}\langle \chi_n(x),\cdot\rangle \,d\mu(x),\quad n\in \mathbb{N}_0.
\end{equation}
\end{defn}


From \cite{BM07} (see also \cite{BKKL99, B03}) follows the following
result.

\begin{thm}\label{t.inf-dim-newton}
 Let $\tau=(\tau_n)_{n=0}^{\infty}\in\mathcal{F}_{\rm{fin}}'({\mathcal D})$ and
 the following three conditions are fulfilled:
 \begin{enumerate}
   \item there exists a $\sigma$-finite measure $\rho$ on $\mathcal{B}(\Gamma_0)$
   such that
   \begin{equation}\label{eq.func-generate-measure}
      \tau(f)=\int_{\Gamma_0}f(\eta)\,d\rho(\eta),\quad
      f\in\mathcal{A}(\mathcal{D})=\mathcal{F}_{\rm{fin}}({\mathcal D}).
   \end{equation}
   \item $\tau$ is $\star$-positive on $\mathcal{A}(\mathcal{D})$, that is
   \begin{equation}\label{eq.positivity-fock-kk}
      \tau(f\star\bar{f})=\int_{\Gamma_0}(f\star\bar{f})(\eta)\,d\rho(\eta)\geq
      0,\quad f\in\mathcal{A}(\mathcal{D}).
   \end{equation}
   \item for every compact $\Lambda\subset\mathbb{R}$ there exists a
   constant $C_{\Lambda}>0$ such that
   \begin{equation}\label{eq.positivity-fock-add}
      \rho(\Gamma_\Lambda^{(n)})\leq C_{\Lambda}^n,\quad
      n\in\mathbb{N},
   \end{equation}
   where
   $\Gamma_\Lambda^{(n)}:=\{\eta\subset\Lambda\,|\,|\eta|=n\}$.
 \end{enumerate}
 Then $\tau$ is a moment functional on ($\mathcal{A}(\mathcal{D}),
 \star$) and the measure $\mu$ in representation (\ref{eq.moment-reprez-infinite-kk})
 is uniquely defined.

 Conversely, every moment functional $\tau\in\mathcal{F}_{\rm{fin}}'({\mathcal D})$ is
 $\star$-positive on $\mathcal{A}(\mathcal{D})$.
\end{thm}

A way of proving this result is similar to that of
Theorem~\ref{t.inf-dim-power} but in this case instead of
(\ref{eq.field-power}) and (\ref{eq.inner-power}) it is necessary to
use the operators
\begin{equation*}
   J(\varphi):\mathcal{A}(\mathcal{D})=\mathcal{F}_{\rm{fin}}({\mathcal D})\to
   \mathcal{A}(\mathcal{D}),\quad
   J(\varphi)f:=(0,\varphi,0,0,\ldots)\star f,
\end{equation*}
and the quasiscalar product
\begin{equation*}
(f,g)_{H_{\tau}}=\tau(f\star \bar{g}),\quad
f,g\in\mathcal{F}_{\rm{fin}}({\mathcal D}).
\end{equation*}

\begin{rem}
It is possible to give a condition on $\rho$ weaker than
(\ref{eq.positivity-fock-add}) which guarantee the existence of the
measure $\mu$ on $\mathcal{D}'$ such that
(\ref{eq.moment-reprez-infinite-kk}) holds, see \cite{BM07} for
details.
\end{rem}

\begin{rem}
Let us explain the connection of the moment problem on
$(\mathcal{A}(\mathcal{D}),\star)$ with some essential objects of
statistical mechanics.

At first we recall the definitions of these objects, see e.g.
\cite{KK02, KKO06} for a detailed explanation. The so-called
$K$-transform maps the functions defined on $\Gamma_0$ into
functions defined on $\Gamma$ by the formula
$$
    K:\mathcal{F}_{\rm{fin}}({\mathcal D})\to \mathcal{P}(\Gamma),\quad
    f\mapsto (Kf)(\gamma):=\sum_{\eta\Subset\gamma}f(\eta),
$$
where the summation is taken over all finite subconfigurations of
$\gamma$ (for short $\eta\Subset\gamma$) and $\mathcal{P}(\Gamma)$
denotes the space of all continuous polynomials on
$\Gamma\subset\mathcal{D}'$. Note here that the $K$-transform
coincides with mapping $I_{\chi}$ from (\ref{eq.bijection-kk}). More
exactly,
\begin{equation}\label{eq.k-trasform-is-fourier}
    (I_{\chi}f)(\gamma)=(Kf)(\gamma),\quad
      f\in\mathcal{F}_{\rm{fin}}({\mathcal D}),\quad \gamma\in\Gamma.
\end{equation}

For a probability measure $\mu$ on $\Gamma$, the so-called
correlation measure $\rho_\mu$ corresponding to $\mu$ is a
$\sigma$-finite measure on $\Gamma_0$ defined by
$$
   \int_{\Gamma_0}f(\eta)\,d\rho_\mu(\eta)
   =\int_{\Gamma}(Kf)(\gamma)\,d\mu(\gamma),\quad
      f\in\mathcal{F}_{\rm{fin}}({\mathcal D}).
$$
If the measure $\rho_\mu$ is absolutely continuous with respect to
the Lebesgue--Poisson measure $\nu$ then the corresponding
Radon--Nikodym derivative
$k_\mu(\eta)=\frac{d\rho_\mu}{d\nu}(\eta),\,\eta\in\Gamma_0$, is
called a correlation functional of a measure $\mu$. Note that in
this case the functions
$$
   k_{\mu,n}(t_1,\ldots,t_n):=
\left\{
  \begin{array}{ll}
    k_{\mu}(\{t_1,\ldots,t_n\}),  & \text{if}\,\, (t_1,\ldots,t_n)\in \widehat{\mathbb{R}}^n\\
    0, & \text{otherwise}
  \end{array},
\right.
$$
are well-known correlation functions of statistical physics, see
e.g. \cite{Ruelle1, Ruelle2}.


In several applications, a $\sigma$-finite measure $\rho$ on
$\Gamma_0$ appears as a given object and the problem is to show that
this $\rho$ can be seen as a correlation measure for a probability
measure on $\Gamma$. Due to (\ref{eq.k-trasform-is-fourier}) it is
easy to see that this problem is a particular case of the moment
problem on $(\mathcal{A}(\mathcal{D}),\star)$. Namely, a given
measure $\rho$ on $\Gamma_0$ is a correlation measure for a
probability measure $\mu$ on $\Gamma$ if and only if the
corresponding functional $\tau$ defined by
(\ref{eq.func-generate-measure}) admits representation
(\ref{eq.moment-reprez-infinite-kk}) with this $\mu$ (i.e., $\tau$
is a moment functional on the algebra ($\mathcal{A},\star$)).

It should be noticed that Theorem~\ref{t.inf-dim-newton} gives the
sufficient conditions that $\rho$ is a correlation measure.  More
exactly, let $\rho$ be a given measure on $\Gamma_0$. Suppose that
the corresponding functional $\tau$ (defined by
(\ref{eq.func-generate-measure})) satisfies all conditions of
Theorem~\ref{t.inf-dim-newton} and, moreover,
$$
    \sum_{n=0}^{\infty}2^n\rho(\Gamma_{\Lambda}^{(n)})<\infty
$$
for every compact $|\Lambda|\subset \mathbb{R}$. Then $\Gamma$ is
the set of full measure $\mu$ and due to equality
(\ref{eq.k-trasform-is-fourier}) and representation
(\ref{eq.moment-reprez-infinite-kk}) we have
$$
   \tau(f)=\int_{\Gamma_0}f(\eta)\,d\rho(\eta)
   =\int_{\Gamma}(Kf)(\gamma)\,d\mu(\gamma),\quad
      f\in\mathcal{F}_{\rm{fin}}({\mathcal D}).
$$
So, $\rho$ is the correlation measure of $\mu$. If, moreover, $\rho$
is absolutely continuous with respect to the Lebesgue--Poisson
measure $\nu$ then the corresponding correlation functional $k_\mu$
of $\mu$ coincides with $\tau$, i.e.,
$\tau=k_\mu=\frac{d\rho}{d\nu}$.
\end{rem}

\begin{rem}
We formulate some  problems, the investigation of which are
essential for the above described theory:


\begin{enumerate}
  \item To give 
        sufficient conditions on a functional
        $\tau=(\tau_n)_{n=0}^{\infty}\in\mathcal{F}_{\rm{fin}}'({\mathcal
        D})$ which would guarantee the existence of representation
        (\ref{eq.moment-reprez-infinite-kk}), i.e., to give
        conditions for validity of Theorem~\ref{t.inf-dim-newton}
        different from (\ref{eq.func-generate-measure}) or (\ref{eq.positivity-fock-add}).
  \item To prove an infinite dimensional analog of
        Theorem~\ref{t.exp.kriterij.Newton} and as a consequence to get an
        analog of Theorem~\ref{t.bogolub.one.dim} for the classical Bogoliubov
        functional. 
  \item To investigate the situation when a measure $\rho$ on
        $\Gamma_0$ is a correlation measure for a probability
        measure  on $\Gamma$. More exactly, to give  sufficient
        conditions on $\rho$ which would guarantee
        the existence of representation (\ref{eq.moment-reprez-infinite-kk}) for the functional
        $\tau\in\mathcal{F}_{\rm{fin}}'({\mathcal D})$ determined by $\rho$
        and, moreover, these conditions should assure that the measure $\mu$
        from (\ref{eq.moment-reprez-infinite-kk}) is concentrated
         on $\Gamma$.
\end{enumerate}
\end{rem}

{\it{Acknowledgments.}}
I would like to thank Prof. Yu.~M.~Berezansky and Dr.~K.~Yusenko
for valuable comments.

\end{document}